\documentclass[11pt, reqno]{article}
\usepackage{amsmath,amssymb,amsfonts,amsthm}
\usepackage{color}

\usepackage[colorlinks=true,linkcolor=blue,citecolor=blue,urlcolor=blue,pdfborder={0 0 0}]{hyperref}
\usepackage{cleveref}

\usepackage{caption}
\usepackage{thmtools}

\usepackage{mathrsfs}

\crefname{theorem}{Theorem}{Theorems}
\crefname{thm}{Theorem}{Theorems}
\crefname{lemma}{Lemma}{Lemmas}
\crefname{lem}{Lemma}{Lemmas}
\crefname{remark}{Remark}{Remarks}
\crefname{prop}{Proposition}{Propositions}
\crefname{defn}{Definition}{Definitions}
\crefname{corollary}{Corollary}{Corollaries}
\crefname{conjecture}{Conjecture}{Conjectures}
\crefname{question}{Question}{Questions}
\crefname{chapter}{Chapter}{Chapters}
\crefname{section}{Section}{Sections}
\crefname{figure}{Figure}{Figures}

\theoremstyle{plain}
\newtheorem{thm}{Theorem}[section]

\newtheorem{lemma}[thm]{Lemma}
\newtheorem{theorem}[thm]{Theorem}

\newtheorem{prop}[thm]{Proposition}

\newtheorem{question}[thm]{Question}
\theoremstyle{definition}

\theoremstyle{remark}
\newtheorem{remark}[thm]{Remark}

\numberwithin{equation}{section}

\renewcommand{\P}{\mathbb P}

\newcommand{\R}{\mathbb R}
\newcommand{\Z}{\mathbb Z}

\newcommand{\cF}{\mathcal F}

\newcommand{\cO}{\mathcal O}

\newcommand{\sA}{\mathscr A}

\newcommand{\sG}{\mathscr G}

\newcommand{\sR}{\mathscr R}

\usepackage[margin=0.9in]{geometry}
\usepackage{extarrows}
\usepackage{titling}
\usepackage{commath}
\usepackage{mathtools}
\usepackage{titlesec}
\newcommand{\eps}{\varepsilon}
\usepackage{bbm}
\usepackage{setspace}
\usepackage{enumitem}

\usepackage[textsize=tiny]{todonotes}

\usepackage{cite}

\newcommand{\Aut}{\operatorname{Aut}}
\newcommand{\Gr}{\operatorname{Gr}}

\newcommand{\bP}{\mathbf P}
\newcommand{\bE}{\mathbf E}

\newcommand{\opleq}{\preccurlyeq}

\def\P{\mathbb{P}}

\DeclareMathSymbol{\leqslant}{\mathalpha}{AMSa}{"36} 
\DeclareMathSymbol{\geqslant}{\mathalpha}{AMSa}{"3E} 
\DeclareMathSymbol{\eset}{\mathalpha}{AMSb}{"3F}     

\renewcommand{\epsilon}{\varepsilon}

\pretitle{\begin{flushleft}\Large}
\posttitle{\par\end{flushleft}\vskip 0.5em
}
\preauthor{\begin{flushleft}}
\postauthor{
\\
\vspace{0.5em}
\footnotesize{The Division of Physics, Mathematics and Astronomy, California Institute of Technology\\ Email: 
\href{mailto:t.hutchcroft@caltech.edu}{t.hutchcroft@caltech.edu}}
\end{flushleft}}
\predate{\begin{flushleft}}
\postdate{\par\end{flushleft}}
\title{\bf Slightly supercritical percolation on nonamenable graphs II: Growth and isoperimetry of infinite clusters
}

\renewenvironment{abstract}
 {\par\noindent\textbf{\abstractname.}\ \ignorespaces}
 {\par\medskip}

\titleformat{\section}
  {\normalfont\large \bf}{\thesection}{0.4em}{}

\author{{\bf Tom Hutchcroft}}
\begin{document}

\date{\small{\today}}

\maketitle

\setstretch{1.1}


\begin{abstract}
We study the growth and isoperimetry of infinite clusters in slightly supercritical  Bernoulli bond percolation on transitive nonamenable graphs under the \emph{$L^2$ boundedness condition} ($p_c<p_{2\to 2}$). Surprisingly, we find that the volume growth of infinite clusters is always purely exponential (that is, the subexponential corrections to growth are bounded) in the regime $p_c<p<p_{2\to 2}$, even when the ambient graph has unbounded corrections to exponential growth. For $p$ slightly larger than $p_c$, we establish the precise estimates
\begin{align*}
\mathbf{E}_p \left[ \# B_\mathrm{int}(v,r) \right] &\asymp \left(r  \wedge \frac{1}{p-p_c} \right)^{\phantom{2}}  e^{\gamma_\mathrm{int}(p) r}
\\
\mathbf{E}_p \left[ \# B_\mathrm{int}(v,r) \mid v \leftrightarrow \infty \right] &\asymp \left(r  \wedge \frac{1}{p-p_c} \right)^2  e^{\gamma_\mathrm{int}(p) r}
\end{align*}
for every $v\in V$, $r \geq 0$, and $p_c < p \leq p_c+\delta$, where the growth rate $\gamma_\mathrm{int}(p) = \lim \frac{1}{r} \log \mathbf{E}_p\#B(v,r)$ satisfies  $\gamma_\mathrm{int}(p) \asymp p-p_c$.
We also 
  prove a percolation analogue of the Kesten-Stigum theorem that holds in the entire supercritical regime and states that the quenched and annealed exponential growth rates of an infinite cluster always coincide. We apply these results together with those of the first paper in this series to prove that the anchored Cheeger constant of every infinite cluster $K$ satisfies \[ \frac{(p-p_c)^2}{\log[1/(p-p_c)]} \preceq \Phi^*(K) \preceq (p-p_c)^2
\]
almost surely for every $p_c<p\leq1$.
\end{abstract}






\section{Introduction}\label{sec:intro}

This paper is the second in a series of three papers analyzing \emph{slightly supercritical percolation} on nonamenable graphs transitive graphs, where the retention parameter $p$ approaches its critical value $p_c$ from above. This regime is typically very difficult to study rigorously, with many of its conjectured features remaining unproven  even for high-dimensional Euclidean lattices where most other regimes are well-understood \cite{hara1994mean,MR1068308,heydenreich2015progress}; see the first paper in this series \cite{hutchcroft2020slightly} for a detailed introduction to the topic.

\medskip
We work primarily under the \emph{$L^2$ boundedness condition} $p_c<p_{2\to 2}$, where $p_{2\to 2}$ 
is the supremal value of $p$ for which the infinite matrix $T_p\in[0,1]^{V\times V}$ defined by $T_p(u,v)=\bP_p(u\leftrightarrow v)$ defines a bounded operator on $\ell^2(V)$. 
This condition, which was introduced in \cite{1804.10191} and developed further in \cite{hutchcroft20192}, is conjectured to hold for every transitive nonamenable graph and proven to hold for various large classes of graphs including highly nonamenable graphs \cite{hutchcroft20192,MR1756965,MR3005730,MR1833805}, Gromov hyperbolic graphs \cite{1804.10191}, and graphs admitting a transitive nonunimodular group of automorphisms \cite{Hutchcroftnonunimodularperc}. In the first paper in this series we proved sharp estimates on the distribution of \emph{finite} clusters near $p_c$ under the $L^2$ boundedness condition. In the present paper we apply these results to analyze the \emph{growth and isoperimetry} of \emph{infinite} clusters. In a forthcoming third paper we will use these results to study the behaviour of \emph{random walk} on infinite slightly supercritical clusters. The present paper can be read independently of \cite{hutchcroft2020slightly} provided that one is willing to take the main results of that paper as a black box. 

\medskip



\textbf{Notation:} We write $\asymp$, $\succeq$, and $\preceq$  to denote equalities and inequalities that hold up to positive multiplicative constants depending only on the graph $G$. For example, ``$f(n) \asymp g(n)$ for every $n\geq 1$'' means that there exist positive constants $c$ and $C$ such that $cg(n) \leq f(n) \leq Cg(n)$  for every $n\geq 1$. We also use Landau's asymptotic notation similarly, so that $f(n)=\Theta(g(n))$  if and only if $f \asymp g$, and $f(n) \preceq g(n)$ if and only if $f(n) = O(g(n))$. Given a matrix $M$ indexed by a countable set $V$, we write $\|M\|_{2\to 2}=\sup\{\|Mf\|_2/\|f\|_2 :f$ a non-zero finitely supported function on $V\}$ for the norm of $M$ considered as an operator on $\ell^2(V)$, which is finite if and only if $M$ extends continuously to a bounded operator on $\ell^2(V)$. We write $\bP_p$ and $\mathbf{E}_p$ for probabilities and expectations taken with respect to the law of Bernoulli-$p$ bond percolation.

\subsection{Expected volume growth}

 We begin by stating our results concerning the volume growth of slightly supercritical clusters. 
We will prove results of two kinds: precise estimates on the \emph{expected} volume of an intrinsic ball for $p$ close to $p_c$, and limit theorems stating that the almost sure volume growth is well-described by its expectation in various senses. This leads to a rather complete description of the volume growth of clusters for percolation with $p_c<p<p_{2\to 2}$, as well as some partial understanding of the remaining supercritical regime $p_{2\to 2} \leq p <1$.

\medskip

We begin with some relevant definitions. Let $G=(V,E)$ be a countable graph. For each $p\in [0,1]$ and $r\geq 0$ we define
\[
\Gr_p(r) := \sup_{v\in V} \bE_p\left[ \# B_\mathrm{int}(v,r)\right],
\]
where $B_\mathrm{int}(v,r)$ denotes the intrinsic ball of radius $r$ around $v$, i.e., the graph distance ball in the cluster of $v$. By a standard abuse of notation we write $B_\mathrm{int}(v,r)$ both for the set of vertices in the ball and the subgraph of the cluster induced by the ball, writing $\#B_\mathrm{int}(v,r)$ for the number of vertices that have intrinsic distance at most $r$ from $v$.
If $G$ has degrees bounded by $M$ then $\Gr_p(r) \leq M (M-1)^{r-1}$ for every $r\geq 1$, so that $\Gr_p(r)$ is finite for every $r\geq 0$.
It is a consequence of Reimer's inequality \cite[Lemma 3.4]{hutchcroft20192} that 
\begin{equation}
\label{eq:Gr_submult_strong}
\bE_p\left[ \# B_\mathrm{int}(v,r+\ell)\right] \leq \bE_p\left[ \# B_\mathrm{int}(v,r-1)\right] +  \bE_p\left[ \# \partial B_\mathrm{int}(v,r)\right] \Gr_p(\ell)
\end{equation}
for every $r,\ell \geq 0$, $p\in [0,1]$, and $v\in V$, and hence that
$\Gr_p(r)$ satisfies the submultiplicative-type inequality
\begin{equation}
\label{eq:Gr_submult}
\Gr_p(r+\ell) \leq \sup\left\{ \bE_p\left[ \# B_\mathrm{int}(v,r-1)\right] +  \bE_p\left[ \# \partial B_\mathrm{int}(v,r)\right] \Gr_p(\ell) :v\in V\right\} \leq \Gr_p(r)\Gr_p(\ell)
\end{equation}
for every $p\in [0,1]$ and $r,\ell \geq 0$. It follows by Fekete's lemma \cite[Appendix II]{grimmett2010percolation} that if $G$ has degrees bounded by $M$ then for each $p\in [0,1]$ there exists $\gamma_\mathrm{int}(p) \in [0,M-1]$ such that
\begin{equation}
\label{eq:gamma_def}
\gamma_\mathrm{int}(p) = \lim_{r\to\infty} \frac{1}{r}  \log \Gr_p(r) = \inf_{r\geq 1} \frac{1}{r}  \log \Gr_p(r).
\end{equation}
Note that 
$\gamma_\mathrm{int}(p)$ is an increasing function of $p$. When $p=1$ we have that $B_\mathrm{int}(v,r)=B(v,r)$ for every $r\geq 0$, so that $\gamma_\mathrm{int}(1)=\gamma(G)$ is simply the exponential growth rate of $G$.  It follows from \eqref{eq:Gr_submult} and \eqref{eq:gamma_def} that for each $p\in [0,1]$ there exists a non-negative, subadditive function $h_p:\{0,1,\ldots\}\to \R$ with $\lim_{r\to \infty}\frac{1}{r}h_p(r)=0$ such that
\begin{equation}
\label{eq:h_def}
 \Gr_p(r) = \exp\left[\gamma_\mathrm{int}(p)r +h_p(r)\right]
\end{equation}
for every $r\geq 0$.  We refer to the function $e^{h_p(r)}=e^{-\gamma_\mathrm{int}(p)r}\Gr_p(r)$ as the \textbf{subexponential correction to growth} for Bernoulli-$p$ percolation on $G$.

\medskip

Our first theorem states that infinite clusters have
%
have purely exponential growth between $p_c$ and $p_{2\to 2}$ in the sense that the subexponential corrections to growth are bounded. 

\begin{theorem}[Bounded subexponential corrections to growth below $p_{2\to2}$]
\label{thm:L2_subexponential}
 Let $G$ be a connected, locally finite, quasi-transitive graph, and let $p_c<p<p_{2\to 2}$. Then there exist positive constants $c_p$ and $C_p$ such that
\begin{equation}
\label{eq:L2_subexponential_annealed}
c_p e^{\gamma_\mathrm{int}(p) r} \leq \bE_p \left[ \# \partial B_\mathrm{int}(v,r) \right] \leq \bE_p \left[ \# B_\mathrm{int}(v,r) \right] \leq C_p e^{\gamma_\mathrm{int}(p) r}
\end{equation}
for every $v\in V$ and $r \geq 0$.
\end{theorem}



We do \emph{not} expect the conclusion of \cref{thm:L2_subexponential} to extend to the entire supercritical phase, even under the assumption of nonamenability. Indeed, the product $T \times \Z^d$ has  $\#B(0,r) = \Theta(r^d e^{\gamma r})$ for appropriate choice of $\gamma$, and it seems plausible that this $r^d$ subexponential correction to growth should also be present in percolation on this graph with $p$ close to $1$. 
 It is therefore an interesting and non-trivial fact that, in our setting, subexponential corrections to growth are always bounded when $p$ is supercritical but not too large.

\medskip

Our next theorem sharpens \cref{thm:L2_subexponential} by giving precise control over the asymptotics of the subexponential corrections to growth when $p$ is close to $p_c$.

\begin{theorem}[Volume growth near criticality]
\label{thm:slightly_subexponential}
Let $G$ be a connected, locally finite, quasi-transitive graph, and suppose that $p_c < p_{2\to 2}$. Then there exists a positive constant $\delta$ such that
\begin{align}
\gamma_\mathrm{int}(p) &\asymp p-p_c,
\label{eq:slightly_subexponential_gamma}
\\
\bE_p \left[ \# B_\mathrm{int}(v,r) \right] &\asymp \left(r  \wedge \frac{1}{p-p_c} \right)^{\phantom{2}}  e^{\gamma_\mathrm{int}(p) r},
\label{eq:slightly_subexponential_unconditioned}
\\
\hspace{-2cm}\text{and}\qquad \bE_p \left[ \# B_\mathrm{int}(v,r) \mid v \leftrightarrow \infty \right] &\asymp \left(r  \wedge \frac{1}{p-p_c} \right)^2  e^{\gamma_\mathrm{int}(p) r}
\label{eq:slightly_subexponential_conditioned}
\end{align}
for every $v\in V$, $r \geq 0$, and $p_c < p \leq p_c+\delta$. 
\end{theorem}


The critical version of this estimate, stating that $\bE_{p_c} \#B_\mathrm{int}(v,r) \asymp r$ for every $r\geq 1$, was proven to hold for any transitive graph satisfying the triangle condition (which is implied by the $L^2$ boundedness condition \cite[p.4]{hutchcroft20192}) in \cite{MR2551766,sapozhnikov2010upper}.
 The transition from critical-like to supercritical-like behaviour outside a scaling window of intrinsic radius $|p-p_c|^{-1}$ is typical of off-critical percolation in high-dimensional settings \cite{hutchcroft2020slightly,hutchcroft2021high,chatterjee2021subcritical}. As is common to such analyses, our proofs will often treat the inside-window and outside-window cases separately, with the inside-window results following straightforwardly from what is known about critical percolation.


\begin{remark}
The estimates of \cref{thm:L2_subexponential,thm:slightly_subexponential} are both significantly stronger than they would be if the right hand sides of \eqref{eq:slightly_subexponential_unconditioned} and \eqref{eq:slightly_subexponential_conditioned} contained terms of the form $e^{\Theta(\gamma_\mathrm{int}(p) r)}$, where the implicit constants in the upper and lower bounds could be different, rather than the exact exponential term $e^{\gamma_\mathrm{int}(p) r}$. In fact it is rather unusual to have such a sharp near-critical estimate in which the exact constant in the exponential is determined, and we are not aware of any other works in which this has been possible. (Indeed, for the near-critical two-point function on the high-dimensional lattice $\Z^d$ the constant in the exponential must be \emph{different} at distances on the order of the correlation length than it is at very large distances, as the equality of exponential rates across these scales would be inconsistent with Ornstein-Zernike decay  at very large scales \cite{hutchcroft2021high}.)
\end{remark}

\begin{remark}
It remains an open problem to establish an analogue of \cref{thm:slightly_subexponential} for infinite slightly supercritical percolation clusters on $\Z^d$ with $d$ large, i.e., to determine the precise manner in which quadratic growth within the scaling window \cite{MR2551766,sapozhnikov2010upper} transitions to $d$-dimensional growth on large scales \cite{MR1404543}. This problem is, in turn, closely related to the problem of computing the asymptotics of the \emph{time constant} for supercritical percolation as $p \downarrow p_c$. An analogous problem for high-dimensional \emph{random interlacements} has recently been solved to within subpolynomial factors in \cite{hernandez2021chemical}, and regularity results for the percolation time constant have been established in \cite{cerf2021time,MR4234130,MR3710798}.
\end{remark}

\subsection{Almost sure volume growth}

Our next theorem, which holds for the entire supercritical regime, shows that $\gamma_\mathrm{int}(p)$ also describes the \emph{almost sure} growth rate of the volume of intrinsic balls in infinite clusters in a rather strong sense. It is an analogue of the Kesten-Stigum theorem for supercritical branching processes \cite{KestenStigum66,LPP95}, and shows that the expectations studied in \cref{thm:L2_subexponential,thm:slightly_subexponential} are indeed the correct quantities to study if one wishes to understand the asymptotic growth of infinite clusters. The proof of this theorem, given in \cref{sec:KestenStigum}, can be read independently of the proofs of \cref{thm:L2_subexponential,thm:slightly_subexponential}.

\begin{theorem}[Expected and almost sure growth rates always coincide]
\label{thm:KestenStigum}
Let $G$ be a connected, locally finite, quasi-transitive graph, let $v$ be a vertex of $G$ and let $p_c < p \leq 1$. Then
\[
 \lim_{r\to\infty} \frac{1}{r}  \log  |\partial B_\mathrm{int}(v,r)| =  \lim_{r\to\infty} \frac{1}{r}  \log  | B_\mathrm{int}(v,r)| = \gamma_\mathrm{int}(p)
\]
$\bP_p$-almost surely on the event that $K_v$ is infinite. Moreover, we also have that
\begin{equation}
\label{eq:KS2}
\liminf_{r\to\infty} e^{-\gamma_\mathrm{int}(p)r}|\partial B_\mathrm{int}(v,r)|>0
\end{equation}
$\bP_p$-almost surely on the event that $K_v$ is infinite.
\end{theorem}

As an aside, we also prove that $\gamma_\mathrm{int}(p)$ is always positive for $p>p_c$ whenever the underlying graph has exponential volume growth. In the nonamenable case this is an easy consequence of the results of, say, \cite{BLS99} or \cite{HermonHutchcroftSupercritical}; we show that a simple and direct proof is also possible in the amenable case.

\begin{theorem}
\label{thm:amenable}
Let $G$ be a connected, locally finite, quasi-transitive graph. If $G$ has exponential volume growth then $\gamma_\mathrm{int}(p)>0$ for every $p_c<p\leq 1$.
\end{theorem}

\begin{remark}
In general, the clusters of invariant percolation processes need not have well-defined rates of exponential growth as shown by Tim\'ar \cite{timar2014stationary}. Interesting recent work of Abert, Fraczyk, and Hayes \cite{abert2022co} has initiated a systematic study of the growth of unimodular random graphs and established criteria in which the growth must exist for unimodular random \emph{trees}. In light of \cref{thm:L2_subexponential,thm:KestenStigum}, Bernoulli percolation may already provide a surprisingly rich test case for this theory.
\end{remark}


\subsection{The anchored Cheeger constant}

Our final set of results concern the \emph{isoperimetry} of the infinite clusters in slightly supercritical percolation.
Recall that the \textbf{anchored Cheeger constant} of a connected, locally finite graph $G$ is defined to be
\[\Phi^*(G) = \liminf_{n\to\infty}\Biggl\{\frac{|\partial_E W|}{\sum_{w\in W} \deg(w)} : v \in W \subseteq V \text{ connected, } n \leq \sum_{w\in W} \deg(w) <\infty\Biggr\},
\]
where $v$ is a fixed vertex of $G$ whose choice does not affect the value of $\Phi^*(G)$. We say that $G$ has \textbf{anchored expansion} if $\Phi^*(G)>0$. This notion was introduced by Benjamini, Lyons, and Schramm \cite{BLS99}, who conjectured that infinite supercritical percolation clusters on nonamenable transitive graphs have anchored expansion. This conjecture was proven in \cite{HermonHutchcroftSupercritical}, following earlier partial results of Chen, Peres, and Pete \cite{chen2004anchored}. The following theorem establishes a quantitative version of this result for graphs satisfying the $L^2$ boundedness condition. Unfortunately we have not quite been able to prove a sharp version of the theorem, but rather are left with a presumably unnecessary logarithmic term in the lower bound.

\begin{theorem}[The anchored Cheeger constant near criticality]
\label{thm:Cheeger}
Let $G$ be a connected, locally finite, quasi-transitive graph with $p_c<p_{2\to 2}$. Then there exist constants $c$ and $C$ such that every infinite cluster in Bernoulli-$p$ bond percolation on $G$ has anchored expansion with anchored Cheeger constant
\[ \frac{c(p-p_c)^2}{\log[1/(p-p_c)]} \leq \Phi^*(K) \leq C(p-p_c)^2
\]
$\bP_p$-almost surely for every $p_c < p \leq 1$.
\end{theorem}

In the forthcoming third paper in this series we prove stronger bounds giving high-probability control of the entire \emph{isoperimetric profile} both for the infinite clusters and their \emph{cores}.

\section{The growth rate near criticality}

In this section we apply the results of \cite{hutchcroft20192} to prove the part of \cref{thm:slightly_subexponential} concerning the limiting exponential growth rate $\gamma_\mathrm{int}(p)=\lim_{r\to\infty}\frac{1}{r} \log \Gr_p(r)$.

\begin{prop}
\label{prop:growth_estimate}
Let $G$ be a connected, locally finite, quasi-transitive graph, and suppose that $p_c<p_{2\to 2}$. Then $\gamma_\mathrm{int}(p) \asymp p-p_c$ for every $p \geq p_c$.
\end{prop}



We begin with the following simple lemma, which is closely related to the results of \cite{sapozhnikov2010upper,MR2551766}. We recall that the \textbf{triangle diagram} $\nabla_p$ is defined by $\nabla_p=\sup_{v\in V} T_p^3(v,v)$, and that a quasi-transitive graph is said to satisfy the \textbf{triangle condition} if $\nabla_{p_c}<\infty$.

\begin{lemma}
\label{lem:growth_in_scaling_window}
Let $G$ be a connected, locally finite, quasi-transitive graph satisfying $\nabla_{p_c} < \infty$. Then 
\begin{align}
\bE_p \left[ \# B_\mathrm{int}(v,r) \right] &\asymp r 
\end{align}
for every $v\in V$, $p_c \leq p \leq 1$, and $1 \leq r \leq (p-p_c)^{-1}$.
\end{lemma}


The proof of this lemma will apply \emph{Russo's formula}, which expresses the derivative of the probability of an increasing event in terms of the expected number of \emph{pivotal edges}; see e.g.\ \cite[Chapter 2]{grimmett2010percolation} for background.  

\begin{proof}[Proof of \cref{lem:growth_in_scaling_window}]
For each $u,v\in V$ and $r\geq 1$, let $\{u \xleftrightarrow{r} v\}=\{u \in B_\mathrm{int}(v,r)\}$ be the event that there exists an open path of length at most $r$ connecting $u$ and $v$. Observe that there are always at most $r$ open pivotals for this event: Indeed, if this event holds and $\gamma$ is an open path of length at most $r$ connecting $v$ to $u$, then any open pivotal for the event must belong to $\gamma$.
As such, summing over $u\in V$ and applying Russo's formula yields that
\[
\frac{d}{dp} \bE_p\#B_\mathrm{int}(v,r) \leq \frac{r}{p} \bE_p\#B_\mathrm{int}(v,r)
\]
for every $r\geq 0$ and $p\in [0,1]$ and hence that
\[
\frac{d}{dp} \log \bE_p\#B_\mathrm{int}(v,r) \leq \frac{r}{p} 
\]
for every $r\geq 0$ and $p\in [0,1]$. Integrating this differential inequality yields that
\begin{equation}
\label{eq:Growth_Russo}
\log \bE_p\#B_\mathrm{int}(v,r) \leq \frac{(p-q)r}{q} + \log \bE_q\#B_\mathrm{int}(v,r)
\end{equation}
for every $r\geq 0$ and $0\leq q \leq p \leq 1$. When $q=p_c$ and $r \leq (p-p_c)^{-1}$ the first term is bounded and we deduce that $\bE_p\#B_\mathrm{int}(v,r) \asymp \bE_{p_c}\#B_\mathrm{int}(v,r)$. 
The claim then follows from the fact that $\bE_{p_c}\#B_\mathrm{int}(v,r) \asymp r$ under  the triangle condition as  established in \cite{MR2551766,sapozhnikov2010upper}. \qedhere


%
\end{proof}






\begin{proof}[Proof of \cref{prop:growth_estimate}]
We begin with the upper bound. It follows from the inequality \eqref{eq:Growth_Russo} that
\begin{equation}
\gamma_\mathrm{int}(p) \leq \frac{p-q}{q} + \gamma_\mathrm{int}(q)
\end{equation}
for every $0 < q < p \leq 1$. Since $\gamma_\mathrm{int}(q)=0$ for every $0<q<p_c$ by sharpness of the phase transition, it follows by taking the limit as $q \uparrow p_c$ that
\begin{equation}
\gamma_\mathrm{int}(p) \leq \frac{p-p_c}{p_c} 
\end{equation}
for every $p_c\leq p \leq 1$. Note that this inequality holds on every connected, locally finite, quasi-transitive graph; the resulting equality $\gamma_\mathrm{int}(p_c)=0$ was already observed to hold for every such graph by Kozma in \cite[Lemma 1]{kozma2011percolation}.


\medskip

We now deduce the lower bound $\gamma_\mathrm{int}(p) \succeq p-p_c$ under the assumption that $p_c<p_{2\to 2}$ from the results of our earlier paper \cite{hutchcroft20192}, which contains both an \emph{extrinsic} version of the same estimate and tools to convert between intrinsic and extrinsic estimates. First, \cite[Corollary 4.3]{hutchcroft20192} gives that
\begin{equation}
\limsup_{\ell\to \infty} \frac{1}{\ell} \log \bE_p \left[ \# K_v \cap B(v,\ell) \right] \asymp p-p_c
\end{equation}
for every $v\in V$ and $p\geq p_c$.
(We only need the lower bound, which is the easier of the two estimates.)
For each $p\in [0,1]$ and $r\geq 0$ we define the matrix $C^\mathrm{int}_{p,r}\in [0,\infty]^{V^2}$ by
\[
C^\mathrm{int}_{p,r}(u,v) = \P_p\left( u \leftrightarrow v, d_\mathrm{int}(u,v)\geq r\right).
\]
The norm of this operator is bounded in \cite[Proposition 3.2]{hutchcroft20192}, which states that
\begin{equation}
\|C^\mathrm{int}_{p,r}\|_{2\to 2} \leq 3 \|T_p\|_{2\to 2} \exp\left[ -\frac{r}{e\|T_p\|_{2\to 2}}\right]
\end{equation}
for every $0 \leq p < p_{2\to 2}$ and $r\geq 0$. We can apply this estimate to deduce by Cauchy-Schwarz that
\begin{align*}
\bE_p \!\left[ \# B(v,\ell) \cap \left(K_v \setminus B_\mathrm{int}(v,r)\right)\right] &= \langle C_{p,r}^\mathrm{int} \mathbbm{1}_v, \mathbbm{1}_{B(v,\ell)} \rangle \leq \|C_{p,r}^\mathrm{int}\|_{2\to2} \|\mathbbm{1}_v\|_2 \| \mathbbm{1}_{B(v,\ell)}\|_2
\\&= \|C_{p,r}^\mathrm{int}\|_{2\to2} \sqrt{\#B(v,\ell)}\leq 3 \|T_p\|_{2\to 2} \exp\left[ -\frac{r}{e\|T_p\|_{2\to 2}}\right]\sqrt{\#B(v,\ell)}
\end{align*}
for every $v\in V$, $\ell,r\geq 0$, and $0 \leq p < p_{2\to2}$. 
Setting $\alpha_p = c/ \|T_p\|_{2\to 2}$ for an appropriately small constant $c$, setting $\ell=\lceil \alpha_p r \rceil $, and taking the limit as $r\to \infty$, we deduce that
\[
\limsup_{r\to \infty} \frac{1}{r} \log \bE_p \left[ \# B(v,\alpha_p r) \cap \left(K_v \setminus B_\mathrm{int}(v,r)\right)\right] < 0
\]
and hence that
\begin{multline*}
\gamma_\mathrm{int}(p) \geq \limsup_{r\to \infty} \frac{1}{r} \log \bE_p \left[ \# B(v,\alpha_p r) \cap B_\mathrm{int}(v,r) \right] \\= \limsup_{r\to \infty} \frac{1}{r} \log \bE_p \left[ \# B(v,\alpha_p r) \cap K_v \right] \asymp \alpha_p (p-p_c)
\end{multline*}
for every $p_c \leq p < p_{2\to 2}$.
 The claim follows since, by the $L^2$ boundedness condition, $\alpha_p$ is bounded away from zero on a neighbourhood of  $p_c$.
\end{proof}



\section{Subexponential corrections to growth in the $L^2$ regime}




We now begin the proof of our results concerning subexponential corrections to growth for slightly supercritical percolation, \cref{thm:L2_subexponential,thm:slightly_subexponential}. Both theorems will be proven via essentially the same method, although the details required to prove \cref{thm:slightly_subexponential} are a little more involved. In fact we will prove a slightly more general version of \cref{thm:L2_subexponential} which may apply at $p_{2\to 2}$ in some examples. We begin by explaining how each of these results can be deduced from a certain generating function estimate which we then prove in \cref{subsec:generating_functions}.

\medskip

We first introduce some relevant definitions.
Recall that a locally finite quasi-transitive graph $G=(V,E)$ is said to satisfy the \textbf{open triangle condition} at $p$ if for every $\eps>0$ there exists $r$ such that $T_p^3(u,v) \leq \eps$ whenever $d(u,v)\geq r$. 
We say that $G$ satisfies the \textbf{modified open triangle condition} at $p$ if
\[
\lim_{k \to \infty} \sup_{v\in V} T_p^2 P^k T_p(v,v) =0,
\]
where $P$ is the transition matrix of simple random walk on $G$.
It is easily seen that any unimodular quasi-transitive graph satisfying the open triangle condition at $p$ also satisfies the modified open triangle condition at $p$. Moreover, we have by Cauchy-Schwarz that
\[
\sup_{v\in V} T_p^2 P^k T_p(v,v) \leq \|P\|_{2\to 2}^k \|T_p\|_{2\to 2}^3
\]
and hence that if $G$ is nonamenable then it satisfies the modified open triangle condition at every $0\leq p < p_{2\to2}$.
Let $[0,p_{2\to 2}) \subseteq I_\nabla \subseteq [0,1)$ be the set of $p$ for which the modified open triangle condition holds. (We believe it is possible to prove that $(0,p_{2\to 2}) \subseteq I_\nabla \subseteq (0,p_{2\to 2}]$ whenever $G$ is a connected, locally finite, quasi-transitive graph, but do not pursue this here.)

\medskip

The following proposition generalizes \cref{thm:L2_subexponential}.

\begin{prop}[Bounded subexponential corrections to growth under the modified open triangle condition]
\label{prop:L2_subexponential}
 Let $G$ be a connected, locally finite, quasi-transitive graph, and let $p>p_c$ be such that $p \in I_\nabla$. Then there exist positive constants $c_p$ and $C_p$ such that
\begin{equation}
\label{eq:L2_subexponential_annealed}
c_p e^{\gamma_\mathrm{int}(p) r} \leq \bE_p \left[ \# \partial B_\mathrm{int}(v,r) \right] \leq \bE_p \left[ \# B_\mathrm{int}(v,r) \right] \leq C_p e^{\gamma_\mathrm{int}(p) r}
\end{equation}
for every $v\in V$ and $r \geq 0$.
\end{prop}





The upper bounds of both \cref{thm:slightly_subexponential} and \cref{prop:L2_subexponential} will be proven by analysis of the generating function $\sG(p,\alpha,u)$ defined by
\[
\sG(p,\alpha,u) = \bE_p \left[\sum_{v \in K_u} \alpha^{d_\mathrm{int}(u,v)} \right] = \sum_{r\geq0} \alpha^r \bE_p\left[\#\partial B_\mathrm{int}(u,r)\right]
\]
for each $p,\alpha \in [0,1]$ and $u \in V$.
Note that if $\alpha<1$ then we can equivalently write
\begin{equation}
\label{eq:sG_by_parts}
\sG(p,\alpha,u)  = (1-\alpha) \sum_{r\geq0} \alpha^r \bE_p\left[\# B_\mathrm{int}(u,r)\right]
\end{equation}
for each $p,\alpha \in [0,1)$, and $u \in V$.
We also write \[\sG_*(p,\alpha)=\inf_{u\in V} \sG(p,\alpha,u)\qquad \text{ and } \qquad \sG^*(p,\alpha)=\sup_{u\in V} \sG(p,\alpha,u)\] for each $p \in [0,1]$ and $\alpha \in [0,1]$. An easy FKG argument yields that if $G$ is connected and quasi-transitive then there exists a constant $C$ such that
\begin{equation}
\label{eq:sGinfsup}
(p \alpha)^{C} \sG^*(p,\alpha) \leq \sG(p,\alpha,u) \leq (p\alpha)^{-C} \sG_*(p,\alpha)
\end{equation}
for every $p,\alpha\in (0,1]$ and $u \in V$.
 It follows from \eqref{eq:gamma_def} that if $G$ is a connected, locally finite, quasi-transitive graph and $p_c < p \leq 1$, then $\sG(p,\alpha,u)<\infty$ if and only if $\gamma_\mathrm{int}(p) < -\log \alpha$. For each $\alpha \geq 0$, we define $p_\alpha = \sup \{ p \in [0,1] : \sG(p,\alpha,u)<\infty $ for every $u\in V \}=\sup \{ p \in [0,1] : \gamma_\mathrm{int}(p)<-\log \alpha \}$. Similarly, for each $0\leq p \leq 1$ we define $\alpha_p$ to be supremal so that $\sG^*(p,\alpha)<\infty$, so that $\alpha_p=e^{-\gamma_\mathrm{int}(p)}$ for $p_c\leq p \leq 1$.

\medskip

We now state our main result regarding this generating function.

\begin{prop}
\label{prop:Gpalpha_estimate}
Let $G$ be a connected, locally finite, quasi-transitive graph. There exists a continuous function $\kappa : I_\nabla \to (0,\infty)$ such that
\[
\sG^*(p,\alpha) \leq \frac{\kappa(p)}{\alpha_p-\alpha}
\]
for every $p \in I_\nabla$ with $p\geq p_c$ and every $\alpha < 1 \wedge \alpha_p$.
\end{prop}

Note in particular that the constant $\kappa(p)$ is bounded in a neighbourhood of $p_c$ when $p_c<p_{2\to 2}$.
We will first show how \cref{thm:slightly_subexponential,prop:L2_subexponential} can be deduced from \cref{prop:Gpalpha_estimate} in \cref{subsec:growth_deduction} before proving 
 \cref{prop:growth_estimate} in \cref{subsec:generating_functions}. 

\subsection{Deduction of \cref{thm:slightly_subexponential,prop:L2_subexponential} from \cref{prop:Gpalpha_estimate}}
\label{subsec:growth_deduction}

In this section we show how \cref{prop:Gpalpha_estimate} can be used to prove \cref{thm:slightly_subexponential} and \cref{prop:L2_subexponential}.
We will apply the following ``Tauberian theorem" that lets us extract pointwise estimates on the growth from the exponentially averaged estimates provided by \cref{prop:Gpalpha_estimate}. The resulting lemma also relies on the submultiplicative-type estimate of \eqref{eq:Gr_submult}  and is similar in spirit to the submultiplicative Tauberian theorem of \cite[Lemma 3.4]{1709.10515}. We will apply this lemma with $\alpha = e^{-1/r} \alpha_p$, so that $\alpha^{-r}\asymp \alpha_p^{-r}$ and $\sG^*(p,\alpha) \preceq r$ by \cref{prop:Gpalpha_estimate}.

\begin{lemma}
\label{lem:Tauberian}
Let $G$ be a connected, quasi-transitive graph with $C$ vertex orbits. Then the inequality
\[
\Gr_p(r) \leq \Gr_p(\lfloor r/2\rfloor)+\frac{4C^2\sG^*(p,\alpha)^2}{r^2(1-\alpha)} \cdot \alpha^{-r} 
\]
holds for every $v\in V$, $r\geq 1$, and $0<\alpha<1$.
\end{lemma}

\begin{proof}[Proof of \cref{lem:Tauberian}]
For each $u,v\in V$, we have by Cauchy-Schwarz that
\begin{multline*}
\sum_{\ell=0}^{r} \sqrt{\alpha^r \bE_p\left[\# B_\mathrm{int}(u,\ell) \right] \bE_p\left[\# \partial B_\mathrm{int}(v,r-\ell) \right]} \\\leq \sqrt{\sum_{\ell=0}^{r} \alpha^\ell \bE_p\left[\# B_\mathrm{int}(u,\ell) \right]}\sqrt{\sum_{\ell=0}^{r} \alpha^{r-\ell} \bE_p\left[\# \partial B_\mathrm{int}(v,r-\ell) \right]}
\leq \frac{1}{\sqrt{1-\alpha}} \cdot \sG^*(p,\alpha),
\end{multline*}
for each $r\geq 1$ and $\alpha>0$, where we used \eqref{eq:sG_by_parts} in the final inequality. Letting $\cO$ be a complete set of orbit representatives for the action of $\Aut(G)$ on $V$, so that $|\cO|=C$, it follows that
\[
\sum_{u\in \cO} \sum_{\ell=0}^{r} \sqrt{\alpha^r \bE_p\left[\# B_\mathrm{int}(u,\ell) \right] \bE_p\left[\# \partial B_\mathrm{int}(v,r-\ell) \right]} \leq \frac{C}{\sqrt{1-\alpha}} \cdot \sG^*(p,\alpha).
\]
Thus, for each $r\geq 1$ there exists an integer $r/2 \leq \ell \leq r$ such that
\[
\sup_{u\in \cO}\sqrt{\alpha^r \bE_p\left[\# B_\mathrm{int}(u,\ell) \right] \bE_p\left[\# \partial B_\mathrm{int}(v,r-\ell) \right]} \leq \frac{2C}{r\sqrt{1-\alpha}} \cdot \sG^*(p,\alpha).
\]
Applying \eqref{eq:Gr_submult} with this choice of $\ell$ it follows that
%
\begin{align*}
\bE_p\left[\# B_\mathrm{int}(v,r)\right] &\leq \bE_p\left[\# B_\mathrm{int}(v,r-\ell) \right]+ \sup_{u\in V}\bE_p\left[\# B_\mathrm{int}(u,\ell) \right] \bE_p\left[\# \partial B_\mathrm{int}(v,r-\ell) \right]
\\
&\leq \bE_p\left[\# B_\mathrm{int}(v,\lfloor r/2\rfloor)\right]+\frac{4C^2 \sG^*(p,\alpha)^2}{r^2(1-\alpha)} \cdot \alpha^{-r} 
\end{align*}
for every $r\geq 1$ and $\alpha>0$ as claimed.
\end{proof}

The proof will also apply the following refinement of Fekete's lemma, 
which lets us relate $\gamma_\mathrm{int}(p)$ directly to the expected size of a \emph{sphere} (rather than to a ball)  when it is positive.
This lemma will be used to establish the lower bounds of both \cref{thm:slightly_subexponential} and \cref{prop:L2_subexponential}.

\begin{lemma}
\label{lem:better_Fekete}
Let $G$ be a connected, locally finite, quasi-transitive graph. There exists a positive constant $C \geq 1$ such that 
\[
\gamma_\mathrm{int}(p)  = \inf_{r\geq 1} \inf_{v\in V} \frac{1}{r}  \log \left(C p^{-C} \bE_p\left[ \# \partial B_\mathrm{int}(v,r)\right] \right)
= \lim_{r \to\infty} \inf_{v\in V} \frac{1}{r} \log  \bE_p\left[ \# \partial B_\mathrm{int}(v,r)\right]
 \]
for every $p \geq p_c$.
In particular, the limit on the right exists for every $p \geq p_c$.
\end{lemma}

\begin{proof}[Proof of \cref{lem:better_Fekete}]
Fix $p\in [0,1]$. 
We trivially have that
\begin{align*}
\inf_{r\geq 1} \inf_{v\in V} \frac{1}{r} \log \left(C p^{-C} \bE_p\left[ \# \partial B_\mathrm{int}(v,r)\right]\right) &\leq \liminf_{r \to\infty} \inf_{v\in V} \frac{1}{r} \log  \bE_p\left[ \# \partial B_\mathrm{int}(v,r)\right]
\\
 &\leq \limsup_{r \to\infty}\sup_{v\in V} \frac{1}{r} \log  \bE_p\left[ \# \partial B_\mathrm{int}(v,r)\right]\\ &\leq 
\lim_{r \to\infty} \frac{1}{r} \log \sup_{v\in V} \bE_p\left[ \#  B_\mathrm{int}(v,r)\right] = \gamma_\mathrm{int}(p)
\end{align*}
for every $C \geq 1$ and $p>0$. 
Thus, it suffices to prove that there exists a constant $C'\geq 1$ such that
\begin{equation}
\label{eq:gamma_sphere_proof}
\gamma_\mathrm{int}(p)  \leq \inf_{r\geq 1} \inf_{v\in V} \frac{1}{r}  \log \left(C' p^{-C'} \bE_p\left[ \# \partial B_\mathrm{int}(v,r)\right] \right)
\end{equation}
 whenever $p \geq p_c$.
 It follows straightforward from quasi-transitivity and the Harris-FKG inequality that there exists a constant $C$ such that 
\begin{equation}
\label{eq:changing_base_growth}
p^C \operatorname{Gr}_p(r) \leq 
\bE_p \left[\# B_\mathrm{int}(v,r+C)\right] \leq 
M^C \bE_p \left[\# B_\mathrm{int}(v,r)\right]
\end{equation}
for every $v\in V$ and $r\geq 1$, where $M$ is the maximum degree of $G$. Substituting this inequality into \eqref{eq:Gr_submult_strong} yields that
\begin{align*}
\bE_p \left[\# B_\mathrm{int}(v,(k+1)r)\right]
 &\leq \bE_p \left[\# B_\mathrm{int}(v,kr)\right]+ \bE_p\left[ \# \partial B_\mathrm{int}(v,r)\right]\Gr_p(kr)
\\
&\leq \bE_p \left[\# B_\mathrm{int}(v,kr)\right]+ \left(\frac{M}{p}\right)^C\bE_p\left[ \# \partial B_\mathrm{int}(v,r)\right] \bE_p \left[\# B_\mathrm{int}(v,kr)\right]
 \end{align*}
for every $r,k \geq 0$, and it follows by induction on $k$ that
\[
\bE_p \left[\# B_\mathrm{int}(v,kr)\right] \leq \left(1+ \left(\frac{M}{p}\right)^{C} \bE_p\left[ \# \partial B_\mathrm{int}(v,r)\right]\right)^{k-1} \bE_p \left[\# B_\mathrm{int}(v,r)\right]
\]
for every $r,k \geq 1$. Since $\bE_p \left[\# B_\mathrm{int}(v,kr)\right]\to \infty$ as $k\to \infty$ when $p\geq p_c$, the claimed inequality \eqref{eq:gamma_sphere_proof} follows easily from this together with a further application of \eqref{eq:changing_base_growth}.
\end{proof}

We are now ready to prove \cref{prop:L2_subexponential} and hence \cref{thm:L2_subexponential}.

\begin{proof}[Proof of \cref{prop:L2_subexponential}]
The lower bound follows immediately from \cref{lem:better_Fekete}. We now prove the upper bound; we will take care to keep track of how the relevant constants blow up as $p \downarrow p_c$ so that the estimates we derive here can also be used in the proof of \cref{thm:slightly_subexponential}. We apply \cref{prop:Gpalpha_estimate} together with \cref{lem:Tauberian} to deduce that there exists a continuous function $\kappa_1:I_\nabla \to (0,\infty)$ such that
\[
\Gr_p(r) \leq \Gr_p(\lfloor r/2 \rfloor)+\frac{4}{r^2(1-\alpha)} \cdot \alpha^{-r} \sG^*(p,\alpha) \leq \Gr_p(\lfloor r/2 \rfloor)+\frac{4\kappa_1(p)^2}{r^2 (\alpha_p-\alpha)^2(1-\alpha)} \cdot \alpha^{-r} 
\]
for every $r\geq 1$, $p\in I_\nabla$ with $p\geq p_c$, and $\alpha < \alpha_p := e^{-\log \gamma_\mathrm{int}(p)}$. Taking $\alpha =  r\alpha_p/(r+1)$ we deduce that
\[
\Gr_p(r) \leq \Gr_p(\lfloor r/2 \rfloor)+4\kappa_1(p)^2\left(\frac{r}{r+1}\right)^{-r}  \frac{r+1}{\alpha_p^2(r+1-r\alpha_p)} \cdot \alpha_p^{-r}
\]
for every $r\geq 1$, $p\in I_\nabla$ with $p\geq p_c$, and $\alpha < \alpha_p$. Since $\alpha_p \geq 1/M$ for every $p$, where $M$ is the maximum degree of $G$, it follows that there exists a continuous function $\kappa_2:I_\nabla\to (0,\infty)$ such that
\begin{equation}
\label{eq:growth_proof_recursion}
\Gr_p(r) \leq \Gr_p(\lfloor r/2 \rfloor)+\kappa_2(p) \frac{r+1}{r+1-r\alpha_p}  \alpha_p^{-r}
\end{equation}
for every $r \geq 1$ and $p\in I_\nabla$ with $p\geq p_c$. Fix $r\geq 1$, let $r_{i+1} = \lfloor r_i /2 \rfloor$ for each $i
\geq 0$, and let $k(r)= \min \{i \geq 1: r_i=0\} \leq \log_2 r$. It follows recursively that
\begin{equation}
\label{eq:growth_proof_coarse}
\Gr_p(r) \leq \kappa_2(p) \frac{r+1}{r+1-r\alpha_p} \sum_{i=0}^{k(r)} \alpha_p^{-r_i } \leq \kappa_2(p)  \frac{r+1}{(r+1-r\alpha_p)(1-\alpha_p)}  \alpha_p^{-r},
\end{equation}
where we used that $(r+1)/(r+1-r\alpha_p)$ is an increasing function of $r$ in the first inequality and bounded $\sum_{i=0}^{k(r)} \alpha_p^{-r_i} \leq \alpha_p^{-r} \sum_{i=0}^\infty \alpha^{i}=(1-\alpha_p)^{-1} \alpha_p^{-r}$ in the second inequality.
(This last bound is rather coarse, and we will need a slightly more refined analysis when we prove \cref{thm:slightly_subexponential}.)
When $p>p_c$ we have by \cref{prop:growth_estimate} that $\alpha_p<1$ so that the prefactor on the right is bounded by a $p$-dependent constant as required.
\end{proof}

We now prove the unconditional growth estimates of \cref{thm:slightly_subexponential} by a slight variation on the proof of \cref{prop:L2_subexponential} above.

\begin{lemma}
\label{lem:unconditioned_growth}
Let $G$ be a connected, locally finite, quasi-transitive graph, and suppose that $p_c < p_{2\to 2}$. Then there exists a positive constant $\delta$ such that
\begin{align}
\bE_p \left[ \# B(v,r) \right] &\asymp \left(r  \wedge \frac{1}{p-p_c} \right)^{\phantom{2}}  e^{\gamma_\mathrm{int}(p) r} 
\end{align}
for every $v\in V$, $r \geq 0$, and $p_c \leq p \leq p_c+\delta$. 
\end{lemma}


\begin{proof}[Proof of \cref{thm:slightly_subexponential}]
It follows from \cref{lem:growth_in_scaling_window,prop:growth_estimate} that the estimate
\begin{equation}
\label{eq:growth_inside_window_restated}
\Gr_p(r) \asymp r \asymp r e^{\gamma_\mathrm{int}(p)r} \qquad \text{ for every $r\leq (p-p_c)^{-1}$}
\end{equation}
holds for every $p\geq p_c$ and $r \geq 1$. Moreover, it follows from \cref{prop:growth_estimate,lem:better_Fekete} and a little elementary analysis that
\[
\Gr_p(r) \geq \sup_{v\in V}\sum_{\ell=0}^r \bE_p\left[ \# \partial B_\mathrm{int}(v,r)\right] 
\succeq \sum_{\ell=0}^r e^{\gamma_\mathrm{int}(p)\ell} \succeq \left(r \vee \frac{1}{p-p_c}\right) e^{\gamma_\mathrm{int}(p)r}
\]
for every $r\geq 1$, so that it remains only to prove the desired upper bounds on $\Gr_p(r)$ in the case that $p>p_c$ and $r\geq (p-p_c)^{-1}$. Similarly to the proof of \cref{prop:L2_subexponential}, we fix $r\geq (p-p_c)^{-1}$ and let $r_{i+1}=\lfloor r_i/2\rfloor$ for each $i\geq 0$, but now define $k(r)=\min\{i \geq 1 : r_i \leq (p-p_c)^{-1}\}$.
With these definitions in hand, we may apply the estimate \eqref{eq:growth_proof_recursion} recursively as before to deduce that
\begin{equation}
\label{eq:near_critical_growth_recursion}
\Gr_p(r) \leq \Gr_p(\lfloor (p-p_c)^{-1} \rfloor) + \kappa_2(p) \frac{r+1}{r+1-r\alpha_p} \sum_{i=0}^{k(i)}\exp\left[\gamma_\mathrm{int}(p)r_i\right],
\end{equation}
where we recall that $\alpha_p = e^{-\gamma_\mathrm{int}(p)}$. We have by \cref{prop:growth_estimate}  that $1-\alpha_p\asymp p-p_c$ forevery  $p>p_c$ and hence that the prefactor multiplying the sum of exponentials in \eqref{eq:near_critical_growth_recursion} satisfies
\begin{equation}
\label{eq:Growth_prefactor_alpha}
\left(\frac{r+1}{r+1-r\alpha_p}\right)^{-1} = 1-(1-\frac{1}{r+1})\alpha_p = \frac{1}{r+1} + (1-\alpha_p) - \frac{1-\alpha_p}{r+1} \asymp (p-p_c) \vee \frac{1}{r}
\end{equation}
for every $p>p_c$ and $r\geq 1$. To control the sum of exponentials itself, we note that for each $0\leq i < k(i)$ we have that $r_i-r_{i+1} \geq r_i/2 \geq (p-p_c)^{-1}/2$. It follows from \cref{prop:growth_estimate} that there exists a positive constant $c$ such that
\[
\exp\left[\gamma_\mathrm{int}(p)r_{i+1}\right] \leq \exp\left[\gamma_\mathrm{int}(p)r_{i}-c\right] 
\]
for every $0\leq i < k(i)$ and hence that
\begin{equation}
\label{eq:sum_of_exponentials}
\sum_{i=0}^{k(i)}\exp\left[\gamma_\mathrm{int}(p)r_i\right] \leq \exp\left[\gamma_\mathrm{int}(p)r\right]  \sum_{i=0}^{k(i)} e^{-ci} \preceq \exp\left[\gamma_\mathrm{int}(p)r\right].
\end{equation}
The claimed upper bound follows by substituting \eqref{eq:growth_inside_window_restated}, \eqref{eq:Growth_prefactor_alpha}, and \eqref{eq:sum_of_exponentials} into \eqref{eq:near_critical_growth_recursion} and using that $\kappa_2(p)$ is bounded on a neighbourhood of $p_c$. \qedhere
\end{proof}

We are now ready to conclude the proof of \cref{thm:slightly_subexponential} given \cref{prop:Gpalpha_estimate}. The proof of the lower bound on the conditional expectation outside the scaling window will make use of the precise control on the tail of the radius of finite slightly supercritical clusters established in \cite[Theorem 1.2]{hutchcroft2020slightly}. The proof will apply the BK inequality  and the attendant notion of the disjoint occurrence $A \circ B$ of two increasing events $A$ and $B$; see e.g.\ \cite[Chapter 2]{grimmett2010percolation} for background.

\begin{proof}[Proof of \cref{thm:slightly_subexponential}]
The estimates of \eqref{eq:slightly_subexponential_gamma} and \eqref{eq:slightly_subexponential_unconditioned} are provided by \cref{prop:growth_estimate} and \cref{lem:unconditioned_growth} respectively, so that it remains only to prove \eqref{eq:slightly_subexponential_conditioned}. Let $\delta>0$ be such that $p_c+\delta<p_{2\to 2}$ and such that \cref{lem:unconditioned_growth} and the results of \cite{hutchcroft2020slightly} hold for every $p_c < p \leq p_c+\delta$, and fix one such $p_c<p<p_c+\delta$. All constants appearing below will be independent of this choice of $p$. (They may \emph{a priori} depend on the choice of $\delta$, but this is not a problem since $\delta$ may be chosen once-and-for-all as a function of the graph.)

\medskip

 We begin with the upper bound. For the `outside the scaling window' case $r \geq (p-p_c)^{-1}$, 
 we simply note that
\begin{equation}
\bE_p\left[  \# B_\mathrm{int}(v,r) \mid v \leftrightarrow \infty \right] \leq \bE_p\left[  \# B_\mathrm{int}(v,r) \right] \bP_p(v\leftrightarrow \infty)^{-1} \preceq (p-p_c)^{-2} e^{\gamma_\mathrm{int}(p)}
\end{equation}
by \eqref{eq:slightly_subexponential_unconditioned} as claimed. We now consider the `inside the scaling window' case $r \leq (p-p_c)^{-1}$.
  Let $u,v\in V$ and $r\geq 1$. By considering the final intersection of some simple open path of length at most $r$ connecting $v$ to $u$ and some infinite simple open path starting at $u$, we see that we have the inclusion of events
\begin{equation}
\{ v \xleftrightarrow{r} u \} \cap \{v\leftrightarrow \infty\} \subseteq \bigcup_{w\in V} \{v \xleftrightarrow{r} w\} \circ \{ w \xleftrightarrow{r} u \} \circ \{ w \leftrightarrow \infty \}.
\end{equation}
Thus, we have by a union bound and the BK inequality that
\begin{equation}
\label{eq:condition_growth_BK}
\bE_p\left[ \mathbbm{1}(v \leftrightarrow \infty) \cdot \# B_\mathrm{int}(v,r)  \right] \leq \sup_{w \in V} \bE_p\left[\# B_\mathrm{int}(v,r)\right]^2 \sup_{w \in V} \bP_p(w \leftrightarrow \infty)
\end{equation}
for every $r\geq 1$. Since $G$ is connected and quasi-transitive we have by the Harris-FKG inequality that there exists a constant $C$ such that 
$\sup_{w\in V} \bP_p(u \leftrightarrow \infty) \leq p^C \inf_{w\in V} \bP_p(u \leftrightarrow \infty)$ for every $p$ and hence that 
\begin{equation}
\label{eq:condition_growth_BK}
\bE_p\left[  \# B_\mathrm{int}(v,r) \mid v \leftrightarrow \infty \right] \preceq \sup_{w \in V} \bE_p\left[\# B_\mathrm{int}(v,r)\right]^2 
\end{equation}
for every $r\geq 1$. This implies the claimed upper bound within the scaling window in conjunction with the upper bound of \eqref{eq:slightly_subexponential_unconditioned}.

\medskip

We now prove the lower bound on the conditional expectation. We begin with the `inside the scaling window' case $r \leq (p-p_c)^{-1}$. Suppose that we explore the cluster of the origin in a breadth-first manner, revealing the status of all edges incident to the intrinsic ball of radius $i$ at step $i$ of the exploration process. Conditional on this exploration up to step $i$, the probability that any vertex in $\partial B_\mathrm{int}(v,i)$ is connected to infinity by an open path that does not visit $B_\mathrm{int}(v,i)$ after its first step is at most $\sup_{w\in V}\bP_p(w \leftrightarrow \infty) \asymp (p-p_c)$. As such, if we define $\cF_i$ to be the $\sigma$-algebra generated by $B_\mathrm{int}(v,i)$ then we have by Markov's inequality that
\begin{equation}
\bP_p\left(v \leftrightarrow \infty \mid \cF_i \right) \preceq  (p-p_c) \cdot \# \partial B_\mathrm{int}(v,i)
\end{equation}
for each $i\geq 0$ and $p_c<p \leq 1$. For each $r\geq 1$ and $\eps>0$, let $T_{r,\eps}=\inf\{i\geq r : \#\partial B_\mathrm{int}(v,i) \leq \eps r \}$, where we set $\inf \emptyset = \infty$. It follows from the above discussion that there exist positive constants $C_1$ and $C_2$ such that
\begin{equation}
\bP_p( v \leftrightarrow \infty \text{ and } T_{r,\eps} \leq 2r \mid \partial B_\mathrm{int}(v,r) \neq \emptyset) \leq C_1 \eps r (p-p_c) \leq C_2 \eps r \bP_p( v \leftrightarrow \infty).
\end{equation}
On the other hand, since $r \leq (p-p_c)^{-1}$, we have by \cite[Lemma 2.1]{hutchcroft2020slightly} that there exists a positive constant $C_3$ such that $\bP_p(\partial B_\mathrm{int}(v,r) \neq \emptyset) \leq C_3 /r$
and hence that
\begin{equation}
\bP_p( T_{r,\eps} \leq 2r \mid v \leftrightarrow \infty ) \leq C_2 C_3 \eps.
\end{equation}
Thus, if we take $\eps = 1/(2C_2C_3)$, we find that $T_{r,\eps}> 2r$ with probability at least $1/2$ on the event that $v$ belongs to an infinite cluster. It follows that
\begin{equation}
\bE_p\left[  \# B_\mathrm{int}(v,2r) \mid v \leftrightarrow \infty \right] \geq \eps r^2 \bP_p(T_{r,\eps}>2r \mid v \leftrightarrow \infty) \geq \frac{r^2}{2C_2C_3},
\end{equation}
for every $r \leq (p-p_c)^{-1}$, which is easily seen to imply the claimed lower bound in this regime.

\medskip

It remains only to prove the lower bound on the conditional expectation in the case $r\geq (p-p_c)^{-1}$. Fix $v\in V$.  Suppose that $y$ and $z$ both belong to $B_\mathrm{int}(v,r)$, and let $\eta_1$ and $\eta_2$ be intrinsic geodesics from $v$ to $y$ and $v$ to $z$ respectively. If $\eta_1$ and $\eta_2$ coincide for the last time at some vertex $x$, then we must have that there exists $0\leq \ell =d_\mathrm{int}(v,x)\leq r$ such that the disjoint occurence $\{ x \in \partial B_\mathrm{int}(v,\ell)\}\circ \{y \in B_\mathrm{int}(x,r-\ell)\} \circ \{z \in B_\mathrm{int}(x,r-\ell)\}$ occurs. Indeed, if we take any three intrinsic geodesics $\gamma_1$, $\gamma_2$, and $\gamma_3$ from $v$ to $x$, $x$ to $y$ and $x$ to $z$ respectively, then the union of $\gamma_1$ with all the closed edges incident to $B_\mathrm{int}(v,\ell)$ is a witness for the event $\{x\in \partial B_\mathrm{int}(x,\ell)\}$, the two paths $\gamma_2$ and $\gamma_3$ are witnesses for the events $\{y \in B_\mathrm{int}(x,r-\ell)\}$ and $\{z \in B_\mathrm{int}(x,r-\ell)\}$, and all three sets are disjoint from each other.
It follows by a union bound that
\begin{equation}
\bE_p \left[ (\# B_\mathrm{int}(v,r))^2 \right] \leq \sum_{\ell=0}^r \sum_{x,y,z} \bP_p \bigl(\{x \in \partial B_\mathrm{int}(u,\ell) \} \circ \{y \in B_\mathrm{int}(x,r-\ell)\} \circ \{z \in B_\mathrm{int}(x,r-\ell)\}\bigr)
\label{eq:intrinsic_2nd_moment_Reimer}
\end{equation} 
for every  $r\geq 0$ and hence by Reimer's inequality that
\begin{align}
\bE_p \left[ (\# B_\mathrm{int}(v,r))^2 \right]&\leq \sum_{\ell=0}^r \bE_p \left[ \# \partial B_\mathrm{int}(u,\ell ) \right] \sup_{v \in V} \bE_p \left[ \# \partial B_\mathrm{int}(v,r-\ell ) \right]^2
\nonumber\\
&\leq \frac{1}{(p-p_c)^2} e^{2\gamma_\mathrm{int}(p)r}\sum_{\ell=0}^r \bE_p \left[ \# \partial B_\mathrm{int}(u,\ell ) \right]  e^{-2\gamma_\mathrm{int}(p)\ell}
\end{align}
 for every $r\geq 0$, where we applied \cref{lem:unconditioned_growth} in the second line.
We may bound the sum appearing here in terms of the generating function $\sG$ and apply \cref{prop:Gpalpha_estimate} to obtain that
\begin{equation}\sum_{\ell=0}^r \bE_p \left[ \# \partial B_\mathrm{int}(u,\ell ) \right]  e^{-2\gamma_\mathrm{int}(p)\ell}
\leq \sG^*(p,\alpha_p^2) \preceq \frac{1}{\alpha_p-\alpha_p^2} \asymp \frac{1}{p-p_c},
\end{equation}
so that 
\begin{equation}
\bE_p \left[ (\#B_\mathrm{int}(v,r))^2 \right]  \preceq \frac{1}{(p-p_c)^3} e^{2\gamma_\mathrm{int}(p)r} \asymp \frac{1}{(p-p_c)} \bE_p \left[ \# B_\mathrm{int}(u,r) \right]^2
\end{equation}
for every $p_c<p \leq 1$ and  $r \geq (p-p_c)^{-1}$. Now, it follows from \cref{lem:unconditioned_growth} that there exists a positive constant $c_1$ such that
\begin{align}
\bE_p \left[ \# \left( B_\mathrm{int}(v,r) \setminus B_\mathrm{int}(v,\lfloor c_1 r\rfloor)\right) \right] &= \bE_p \left[ \# B_\mathrm{int}(v,r) \right] - \bE_p\left[\# B_\mathrm{int}(v,\lfloor c_1 r\rfloor) \right]
\nonumber\\
& \succeq \bE_p \left[ \#B_\mathrm{int}(v,r) \right] \asymp \left(r\wedge \frac{1}{p-p_c}\right) e^{\gamma_\mathrm{int}(p) r}
\end{align}
for every $r\geq 1$. Letting $Z_r = \#(B_\mathrm{int}(v,r) \setminus B_\mathrm{int}(v,\lfloor c_1 r\rfloor))$, we conclude that if $r\geq (p-p_c)^{-1}$ then
\begin{equation}
\bE_p Z_r \asymp \frac{1}{p-p_c} e^{\gamma_\mathrm{int}(p) r} \qquad \text{ and } \qquad \bE_p \left[Z_r^2\right] \preceq \frac{1}{(p-p_c)^3} e^{2\gamma_\mathrm{int}(p) r} \asymp \frac{1}{p-p_c} \left(\bE_p Z_r \right)^2.
\end{equation} 
Since the random variable $Z_r$ is non-zero if and only if $K_v$ has intrinsic radius at least $c_1 r$, it follows from \cite[Theorem 1.2]{hutchcroft2020slightly} that there exist positive constants $C_4$ and $c_2$ such that if $r\geq (p-p_c)^{-1}$ then
\begin{equation}
\bP_p(Z_r>0, v\nleftrightarrow \infty) \leq C_4(p-p_c) e^{-c_2 (p-p_c) r}.
\end{equation}
As such, we have by Cauchy-Schwarz that there exists a constant $C_5$ such that
\begin{equation}
\bE_p \left[Z_r \mathbbm{1}(Z_r >0, v \nleftrightarrow \infty)\right] \leq \sqrt{\bE_p \left[Z_r^2\right] \bP_p(Z_r >0, v \nleftrightarrow \infty)} \leq C_5 e^{-c_2(p-p_c)r} \bE_p Z_r
\end{equation}
for every $r\geq 1$. It follows that there exists a constant $C_6$ such that if $r\geq C_6 (p-p_c)^{-1}$ then 
\begin{equation}\bE_p \left[Z_r \mathbbm{1}(Z_r >0, v \nleftrightarrow \infty)\right] \leq \frac{1}{2}\bE_p Z_r \end{equation}
and hence
\begin{equation}
\bE_p \left[Z_r \mathbbm{1}(v \leftrightarrow \infty)\right] = \bE_p \left[Z_r\right]-\bE_p \left[Z_r \mathbbm{1}(Z_r >0, v \nleftrightarrow \infty)\right] \geq \frac{1}{2}\bE_p\left[ Z_r \right] \asymp \frac{1}{p-p_c} e^{\gamma_\mathrm{int}(p)r}.
\end{equation}
It follows that
\begin{equation}
\bE_p\left[\#B_\mathrm{int}(v,r)\mid v \leftrightarrow \infty\right] \geq \bE_p \left[Z_r 
\mid v \leftrightarrow \infty\right] \succeq \frac{1}{(p-p_c)^2} e^{\gamma_\mathrm{int}(p)r}.
\end{equation}
for every $r\geq C_6 (p-p_c)^{-1}$. This is easily seen to conclude the proof since the remaining cases $(p-p_c)^{-1} < r < C_6 (p-p_c)^{-1}$ can be handled by monotonicity in $r$.
\end{proof}

\begin{remark}
The proof of \cref{thm:slightly_subexponential} also yields that there exists $\delta>0$ such that
\begin{equation}
\bE_p \left[(\# B_\mathrm{int}(v,r))^2 \mid v \leftrightarrow \infty \right] \asymp \bE_p \left[\# B_\mathrm{int}(v,r) \mid v \leftrightarrow \infty \right]^2 \asymp \left(r \wedge \frac{1}{p-p_c}\right)^2 e^{2\gamma_\mathrm{int}(p)r}
\end{equation}
for every $p_c < p \leq p_c+\delta$ and $r\geq 1$, and hence that $\# B_\mathrm{int}(v,r)$ is of order $(r \wedge (p-p_c)^{-1})e^{\gamma_\mathrm{int}(p)r}$ with good probability conditioned on $\{v\leftrightarrow \infty\}$ for each $r\geq 1$. It should be possible to prove similar estimates for higher moments with a little further work.
\end{remark}

\subsection{Proof of \cref{prop:Gpalpha_estimate}}
\label{subsec:generating_functions}

In this section we prove \cref{prop:Gpalpha_estimate} and thereby complete the proofs of \cref{thm:L2_subexponential,thm:slightly_subexponential,prop:L2_subexponential}. Our proof will work by deriving and analyzing a certain differential inequality concerning the generating function $\sG(p,\alpha,u)$. To this end, we define for each $p\in [0,1]$, $\alpha>0$, and $u\in V$ the formal derivative
\[
\frac{\partial}{\partial \alpha} \sG(p,\alpha,u) := \bE_p\left[ \sum_{v \in K_u} d_\mathrm{int}(u,v) \alpha^{d_\mathrm{int}(u,v)-1} \right].
\]
Note that, being defined as a convergent power series, $\sG(p,\alpha,u)$ is an analytic function of $\alpha$ with derivative 
$\frac{\partial}{\partial \alpha} \sG(p,\alpha,u)$
on $(0,\alpha_p)$ for each $p\in [0,1]$ and $u\in V$. We will deduce \cref{prop:Gpalpha_estimate} from the following differential inequality.

\begin{prop}
\label{prop:alpha_dif_ineq}
Let $G$ be a connected, locally finite, quasi-transitive graph. Then there exists a continuous function $\eta: I_\nabla \times (0,1] \to (0,1]$  such that
\[
\frac{\partial}{\partial \alpha} \sG(p,\alpha,u)  \geq \eta(p,\alpha) \sG(p,\alpha,u)^2
\]
for every $p\in [0,1]$, $u\in V$, and $0< \alpha \leq 1$ such that $\sG_{p,\alpha}^*<\infty$.
\end{prop}

\begin{proof}[Proof of \cref{prop:Gpalpha_estimate} given \cref{prop:alpha_dif_ineq}]
Note that $\alpha_p \geq 1/(M-1) >0$ for every $p\in [0,1]$, where $M$ is the maximum degree of $G$.
Fix $p_c \leq p \in I_\nabla$. 
It follows from \cref{lem:better_Fekete} that $\sG_*(p,\alpha_p)=\sG^*(p,\alpha_p)=\infty$ and $\sG^*(p,\alpha)<\infty$ for every $\alpha < \alpha_p$. (Since $\sG(p,\alpha,v)$ can be written as a power series in $\alpha$ with non-negative coefficients and with radius of convergence $\alpha_p$, this conclusion may also be derived from the Vivanti–Pringsheim theorem.) The differential inequality of \cref{prop:alpha_dif_ineq} implies that
\[
\frac{\partial}{\partial \alpha} \sG(p,\alpha,u)^{-1}
= - \sG(p,\alpha,u)^{-2} \frac{\partial}{\partial \alpha} \sG(p,\alpha,u)^{-1} \leq - \eta(p,\alpha) \leq - \inf \{ \eta(p,\beta):\beta \in [1/2M,1]\}
\]
for every $u\in V$ and $\alpha_p /2 \leq \alpha<\alpha_p$.
Integrating this differential inequality yields that
\[
\sG(p,\alpha,u)^{-1} \geq  (\alpha_p-\alpha) \inf \{ \eta(p,\beta):\beta \in [1/2M,1]\}
\]
for every $\alpha_p/2 \leq \alpha < \alpha_p$ and $u\in V$, and the claim follows by rearranging.
\end{proof}

We now begin to work towards the proof of \cref{prop:alpha_dif_ineq}. We begin by proving the following lemma, which can be thought of as a `well-separated' version of the same inequality.

\begin{lemma}
\label{lem:ultraviolet}
Let $G$ be a countable graph, and let $P$ be the transition matrix of simple random walk on $G$. Then
\[
\sum_{v,w \in V} P^k(v,w) \bE_p\left[ \sum_{y\in V} \mathbbm{1} (v \nleftrightarrow w) \alpha^{d_\mathrm{int}(u,v) + d_\mathrm{int}(w,y)}\right] \geq \sG_*(p,\alpha)^2 -  \sG^*(p,\alpha)^2 \sup_{a\in V}\left[T_p^2 P^k T_p\right](a,a).
\]
for every $p\in [0,1]$, $u\in V$, and $0< \alpha \leq 1$ such that $\sG^*(p,\alpha)<\infty$.
\end{lemma}

The proof of this lemma (along with the general strategy of proving a differential inequality for percolation by first proving a well-separated variant on the same inequality) is adapted from proofs of similar statements concerning the the $\alpha=1$ case, such as that of \cite[Lemma 3.2]{MR2551766} and \cite[Section 5]{hutchcroft20192}; the basic idea is ultimately due to Aizenman and Newman \cite{MR762034}. 

\begin{proof}[Proof of \cref{lem:ultraviolet}]
We prove the estimate in the case $\alpha<1$, which is the case we are primarily interested in. The case $\alpha =1$ is simpler, and is very similar to arguments already in the literature such those appearing in \cite[Section 5]{hutchcroft20192}. (Moreover, when $G$ is quasi-transitive, the case $\alpha=1$ can be deduced from the case $\alpha<1$ by taking the limit as $\alpha \uparrow 1$.) Fix $u\in V$ and $0< \alpha <1$. Writing $\{x \xleftrightarrow{r} y\}$ for the event that $x$ and $y$ are connected by an open path of length at most $r$, we have that
\begin{multline*}
\sum_{v,w \in V} P^k(v,w) \bE_p\left[ \sum_{y\in V} \mathbbm{1} (v \nleftrightarrow w) \alpha^{d_\mathrm{int}(u,v) + d_\mathrm{int}(w,y)}\right] \\=
(1-\alpha)^2 \sum_{v,w,y \in V} \sum_{r_1,r_2 \geq 0} \alpha^{r_1+r_2} P^k(v,w)  \bP_p\left( v \nleftrightarrow w, u \xleftrightarrow{r_1} v, w \xleftrightarrow{r_2} y\right).
\end{multline*}
Observe that for each $u,v,w,y \in V$ and $r_1,r_2 \geq 0$ we have that
\[
\bP_p\left( v \nleftrightarrow w, u \xleftrightarrow{r_1} v, w \xleftrightarrow{r_2} y\mid K_u\right) = \mathbbm{1}(u \xleftrightarrow{r_1} v, w \notin K_u) \bP_p(w \xleftrightarrow{r_2} y \text{ off }K_u \mid K_u),
\]
where we write $\{
x \xleftrightarrow{r} y \text{ off } A \}$ for the event that $x$ and $y$ are connected by an open path of length at most $r$ that does not visit any vertices of $A$, including at its endpoints.
Define $Q_{r}(a,b;A) = \mathbbm{1}(a \notin A)\bP_p(a \xleftrightarrow{r} b \text{ off } A )$ for each $a,b \in V$, $A \subseteq V$, and $r \geq 0$. Since the event $\{w \xleftrightarrow{r_2} y$ off $K_u\}$ is conditionally independent given $K_u$ of the status of any edge both of whose endpoints belong to $K_u$, we have that
\begin{equation}
\label{eq:ultraviolet1}
\bP_p\left( v \nleftrightarrow w, u \xleftrightarrow{r_1} v, w \xleftrightarrow{r_2} y\mid K_u\right) = \mathbbm{1}(u \xleftrightarrow{r_1} v) Q_{p,r_2}(w,y;K_u)
\end{equation}
for every $u,v,w,y \in V$ and $r_1,r_2 \geq 0$.

We now apply a standard argument similar to that appearing in the proof of \cite[Lemma 3.2]{MR2551766} to prove that 
\begin{equation}
\label{eq:ultravioletQ}
Q_{r,A}(a,b):=\mathbbm{1}(a \notin A)\bP_p(a \xleftrightarrow{r} b \text{ off } A ) \geq \bP_{p}(a \xleftrightarrow{r} b)-\sum_{x \in A} \bP_p(a \xleftrightarrow{} x)\bP_p(x \xleftrightarrow{r} b)
\end{equation}
for every $a,b\in V$, $A \subseteq V$, an $r\geq 0$. Fix such an $a,b,A$, and $r$. The inequality holds trivially if $a \in A$, so suppose not. In this case, we have that
\[
\bP_p(a \xleftrightarrow{r} b \text{ off } A )=\bP_p(a \xleftrightarrow{r} b) - \bP_p(a \xleftrightarrow{r} b \text{ only via } A ),
\]
where we write $\{a \xleftrightarrow{r} b \text{ only via } A\}$ for the event that $a$ is connected to $b$ by a simple open path of length at most $r$ and every such path passes through $A$. Next observe that 
\[
\{a \xleftrightarrow{r} b \text{ only via } A\} \subseteq \bigcup_{x \in A} \{a \leftrightarrow x\}\circ\{x \xleftrightarrow{r} b\}.
\]
Indeed, if $\gamma$ is a simple open path of length at most $r$ from $a$ to $b$ that visits $A$ at some vertex $x$, then the portions of $\gamma$ before and after visiting $x$ are disjoint witnesses for the events $\{a \leftrightarrow x\}$ and $\{x \xleftrightarrow{r} b\}$. The claimed inequality \eqref{eq:ultravioletQ} follows by applying the union bound and the BK inequality.
Putting the estimates \eqref{eq:ultraviolet1} and \eqref{eq:ultravioletQ} together, we deduce that 
\begin{multline*}
\bP_p\left( v \nleftrightarrow w, u \xleftrightarrow{r_1} v, w \xleftrightarrow{r_2} y\mid K_u\right) = \mathbbm{1}(u \xleftrightarrow{r_1} v) \bP_{p}(w \xleftrightarrow{r_2} y)\\-\sum_{x \in K_u} \mathbbm{1}(u \xleftrightarrow{r_1} v, w \notin K_u) \bP_p(w \xleftrightarrow{} x)\bP_p(x \xleftrightarrow{r_2} y),
\end{multline*}
and hence that
\begin{multline*}
\bP_p\left( v \nleftrightarrow w, u \xleftrightarrow{r_1} v, w \xleftrightarrow{r_2} y \right) \geq \bP_p(u \xleftrightarrow{r_1} v) \bP_{p}(w \xleftrightarrow{r_2} y)
\\-\sum_{x\in V} \bP_p(u \xleftrightarrow{r_1} v, u \leftrightarrow x)  \bP_p(w \xleftrightarrow{} x)\bP_p(x \xleftrightarrow{r_2} y)
\end{multline*}
for every $v,w,y \in V$ and $r_1,r_2 \geq 0$.
Summing over $v,w,y \in V$ and $r_1,r_2 \geq 0$ yields that
\begin{align*}
\sum_{v,w \in V} P^k&(v,w) \bE_p\left[ \sum_{y\in V} \mathbbm{1} (v \nleftrightarrow w) \alpha^{d_\mathrm{int}(u,v) + d_\mathrm{int}(w,y)}\right]\\
&= (1-\alpha)^2 \sum_{v,w,y \in V} \sum_{r_1,r_2 \geq 0} \alpha^{r_1+r_2} P^k(v,w)  \bP_p\left( v \nleftrightarrow w, u \xleftrightarrow{r_1} v, w \xleftrightarrow{r_2} y\right)
\\
&\geq (1-\alpha)^2\sum_{v,w,y \in V} \sum_{r_1,r_2 \geq 0} \alpha^{r_1+r_2} P^k(v,w) \bP_p(u \xleftrightarrow{r_1} v) \bP_p(w \xleftrightarrow{r_2} y)
\\
&\hspace{2em}-(1-\alpha)^2\sum_{v,w,y,x \in V} \sum_{r_1,r_2 \geq 0} \alpha^{r_1+r_2} P^k(v,w)
\bP_p(u \xleftrightarrow{r_1} v, u \leftrightarrow x)  \bP_p(w \xleftrightarrow{} x)\bP_p(x \xleftrightarrow{r_2} y)
\\
&\geq \sG_*(p,\alpha)^2 - (1-\alpha) \sum_{v,w,x \in V} \sum_{r_1\geq0} \alpha^{r_1} P^k(v,w)
\bP_p(u \xleftrightarrow{r_1} v, u \leftrightarrow x)  \bP_p(w \xleftrightarrow{} x)\sG^*(p,\alpha)
\end{align*}
for every $v\in V$, $p\in [0,1]$, and $0 < \alpha < 1$ such that the second term on the right  of the last line is finite.
To control this second term, first note that a standard BK inequality argument yields that
\[
\bP_p(u \xleftrightarrow{r_1} v, u \leftrightarrow x) \leq \sum_{a\in V} \bP_p(u \xleftrightarrow{r_1} a) \bP_p(a \leftrightarrow v) \bP_p(a \leftrightarrow x)
\]
for every $u,v,x \in V$ and $r_1 \geq 0$, so that
\begin{multline*}
(1-\alpha) \sum_{v,w,x \in V} \sum_{r_1\geq0} \alpha^{r_1} P^k(v,w)
\bP_p(u \xleftrightarrow{r_1} v, u \leftrightarrow x)  \bP_p(w \xleftrightarrow{} x) 
\\\leq 
(1-\alpha) \sum_{a\in V} \sum_{r_1\geq0} \alpha^{r_1}\bP_p(u \xleftrightarrow{r_1} a) \sum_{v,w,x \in V}  P^k(v,w)
 \bP_p(a \leftrightarrow v) \bP_p(a \leftrightarrow x)  \bP_p(w \xleftrightarrow{} x)
 \\=
 (1-\alpha) \sum_{a\in V} \sum_{r_1\geq0} \alpha^{r_1}\bP_p(u \xleftrightarrow{r_1} a) \left[T_p^2 P^k T_p\right](a,a) \leq \sG^*(p,\alpha)\sup_{a\in V}\left[T_p^2 P^k T_p\right](a,a) 
\end{multline*}
for every $v\in V$, $p\in [0,1]$, and $0 < \alpha < 1$.
Putting everything together, we get that
\[
\sum_{v,w \in V} P^k(v,w) \bE_p\left[ \sum_{y\in V} \mathbbm{1} (v \nleftrightarrow w) \alpha^{d_\mathrm{int}(u,v) + d_\mathrm{int}(w,y)}\right] \geq \sG_*(p,\alpha)^2 -  \sG^*(p,\alpha)^2 \sup_{a\in V}\left[T_p^2 P^k T_p\right](a,a).
\]
for every $u\in V$, $p\in [0,1]$ and $0 \leq \alpha <1$ such that $\sG^*(p,\alpha)<\infty$, as claimed. (It may seem that we need to assume that $\sG^*(p,\alpha)^2 \sup_{a\in V}\left[T_p^2 P^k T_p\right](a,a)$ is finite, but in fact the inequality is trivial if $\sG^*(p,\alpha)^2 \sup_{a\in V}\left[T_p^2 P^k T_p\right](a,a)$ is infinite and $\sG_*(p,\alpha) \leq \sG^*(p,\alpha)$ is not.)
\end{proof}

We now deduce \cref{prop:alpha_dif_ineq} from \cref{lem:ultraviolet}.

\begin{proof}[Proof of \cref{prop:alpha_dif_ineq}]
For each $u,v,w,y \in V$, let $\gamma=\gamma_{v,w}$ be a geodesic from $v$ to $w$ in $G$, and let $\sA(u,v,w,y)$ be the event that that $u,v,w,$ and $y$ all belong to the same cluster, that every edge of $\gamma$ is open, and that every open path from $u$ to $y$ passes through a vertex of $\gamma$. Since $|\gamma|\leq k$ when $P^k(v,w)>0$, we have that
\begin{multline*}
\sum_{v,w,y \in V}P^k(v,w)\bE_p\left[ \alpha^{d_\mathrm{int}(u,y)} \mathbbm{1}(\sA(u,v,w,y))\right] 
\\
\geq \alpha^k \sum_{v,w,y \in V}P^k(v,w)\bE_p\left[ \alpha^{d_\mathrm{int}(u,v) + d_\mathrm{int}(w,y)} \mathbbm{1}(\sA(u,v,w,y))\right].
\end{multline*}
We claim furthermore that
\begin{multline}
\label{eq:sA_open_claim}
\alpha^k \sum_{v,w,y \in V}P^k(v,w)\bE_p\left[ \alpha^{d_\mathrm{int}(u,v) + d_\mathrm{int}(w,y)} \mathbbm{1}(\sA(u,v,w,y))\right]\\
\geq (p\alpha)^k \sum_{u,w,y \in V}P^k(v,w)\bE_p\left[ \alpha^{d_\mathrm{int}(u,v) + d_\mathrm{int}(w,y)} \mathbbm{1}(v \nleftrightarrow w)\right].
\end{multline}
Indeed, let $\omega$ and $\omega'$ be two independent instances of Bernoulli-$p$ bond percolation, and let $\omega''$ be defined by letting $\omega''(e)=\omega'(e)$ if $e$ is traversed by $\gamma$ and by letting $\omega''(e)= \omega(e)$ otherwise, so that $\omega''$ is also distributed as Bernoulli-$p$ bond percolation. 
Condition on $\omega$, and suppose that the event $\{u \leftrightarrow v, v \nleftrightarrow w, w \leftrightarrow y\}$ holds for $\omega$. The conditional probability that every edge $e$ traversed by $\gamma$ is $\omega''$-open is $p^{d(v,w)} \geq p^k$, and on this event the event $\sA(u,v,w,y)$ holds for $\omega''$. Moreover, on this event we have that $\omega'' \geq \omega$ and hence that all intrinsic distances are smaller in $\omega''$ than in $\omega$, so that the claimed inequality follows easily.

Let $u,y \in V$. Suppose that $u$ and $y$ belong to the same cluster, let $\eta$ be an intrinsic geodesic from $u$ to $y$, and let $\eta_i$ be the $i$th vertex visited by $\eta$. Then we have the coarse bounds
\begin{align}
\sum_{v,w\in V} P^k(u,w)\mathbbm{1}(\sA(u,v,w,y)) &\leq \sum_{i=0}^{d_\mathrm{int}(u,y)} \sum_{v,w \in V} \mathbbm{1}(\gamma_{v,w} \text{ visits } \eta_i) P^k(u,w)
\nonumber\\
&\leq \sum_{i=0}^{d_\mathrm{int}(u,y)} |B(\eta_i,k)|^2 \leq d_\mathrm{int}(v,y) \sup_{v\in V} |B(v,k)|^2.
\end{align}
Taking expectations and rearranging, it follows that
\[
\bE_p\left[ d_\mathrm{int}(u,y) \alpha^{d_\mathrm{int}(u,y)}\right] \geq \left[\sup_{v\in V} |B(v,k)|^2\right]^{-1} \sum_{v,w \in V}P^k(v,w)\bE_p\left[ \alpha^{d_\mathrm{int}(u,y)} \mathbbm{1}(\sA(u,v,w,y))\right] 
\]
for every $u,y \in V$, and hence by \eqref{eq:sA_open_claim} that
\begin{multline*}
\bE_p\left[ \sum_{y\in V} d_\mathrm{int}(u,y) \alpha^{d_\mathrm{int}(u,y)}\right]
\\\geq 
(p\alpha)^k \left[\sup_{v\in V} |B(v,k)|^2\right]^{-1}\sum_{u,w,y \in V}P^k(v,w)\bE_p\left[ \alpha^{d_\mathrm{int}(u,v) + d_\mathrm{int}(w,y)} \mathbbm{1}(v \nleftrightarrow w)\right].
\end{multline*}
Applying \cref{lem:ultraviolet}, we obtain that
\begin{equation*}
\frac{\partial}{\partial \alpha} \sG(p,\alpha,u) \geq (p \alpha)^k \left[\sup_{v\in V} |B(v,k)|^2\right]^{-1}\left[1-  (p\alpha)^{-C} \sup_{a\in V}\left[T_p^2 P^k T_p\right](a,a)\right]\sG^*(p,\alpha)^2.
\end{equation*}
It follows by definition of the open triangle condition that there exists $k(p,\alpha)$, bounded on compact subsets of $I_\nabla \times (0,1]$, such that
\[
(p\alpha)^{-C} \sup_{a\in V}\left[T_p^2 P^{k(p,\alpha)} T_p\right](a,a) \leq \frac{1}{2}
\]
and hence that
\begin{equation*}
\frac{\partial}{\partial \alpha} \sG(p,\alpha,u) \geq \frac{1}{2} (p \alpha)^{k(p,\alpha)} \left[\sup_{v\in V} \left|B\left(v,k\left(p,\alpha\right)\right)\right|^2\right]^{-1}\sG^*(p,\alpha)^2 
\end{equation*}
for every $p \in I_\nabla$ and $\alpha \in (0,1]$ such that $\sG^*(p,\alpha)<\infty$. This is easily seen to imply the claim.
\end{proof}


\section{Expected and almost sure growth rates coincide}
\label{sec:KestenStigum}

In this section we prove \cref{thm:KestenStigum}, which states that the expected and almost sure exponential growth rates of an infinite cluster always coincide.
Note that an easier proof of this theorem is possible in the case $p_c < p < p_{2\to 2}$ by applying \cref{thm:L2_subexponential}; in the general supercritical case we have to contend with the possibility that the subexponential corrections to growth are unbounded, which make the second moment calculations more involved.


\begin{proof}[Proof of \cref{thm:KestenStigum}]
In contrast to the rest of the paper, we will allow all the constants appearing in this proof to depend on $p$. 
The almost sure upper bound 
\begin{equation}
\limsup_{n\to\infty} \frac{1}{r}  \log  |\partial B_\mathrm{int}(v,r)| \leq \limsup_{n\to\infty} \frac{1}{r}  \log  |B_\mathrm{int}(v,r)| \leq \gamma_\mathrm{int}(p)
\end{equation}
follows immediately from Markov's inequality and Borel-Cantelli. Thus, to prove the theorem it suffices to prove that the event
\[
\sA_v = \left\{\liminf_{r\to\infty} e^{-\gamma_\mathrm{int}(p)r}|\partial B_\mathrm{int}(v,r)|>0\right\}
\]
satisfies $\bP_p(\sA_v \mid v\to \infty)=1$ for every $v\in V$. This claim is trivial when $\gamma_\mathrm{int}(p)=0$, so we may assume that it is positive. 
It is a consequence of the indistinguishability theorem of H\"aggstr\"om, Peres, and Schonmann \cite[Theorem 4.1.6]{HPS99} that $\bP_p(\sA_v \mid v \to \infty)$ belongs to $\{0,1\}$ and does not depend on $v$, so that it suffices to prove that $\bP_p(\sA_v \mid v \to \infty)>0$ for some $v$. (If $G$ is unimodular then one can alternatively use the indistinguishability theorem of Lyons and Schramm \cite{LS99} in this argument to achieve the same effect.)

\medskip

Let $h_p$, defined by $e^{h_p(r)}=e^{-\gamma_\mathrm{int}(p) r} \sup_{v\in V}\bE_p |B_\mathrm{int}(v,r)|$ for each $r\geq 0$,  describe the subexponential correction to growth of the expected cluster size as in \eqref{eq:h_def}. Let $\lambda \geq 1$ and consider the set 
\[\mathscr{R}_\lambda
=\left\{r \geq 0 : h_p(r) \geq \max_{0\leq \ell \leq r} h_p(\ell) -\lambda 
\, \text{ and } \,
\sup_{v\in V} \bE_p \left[ |B_\mathrm{int}(v,r)|^2 \right] \leq \lambda \sup_{v\in V} \bE_p \left[ |B_\mathrm{int}(v,r)| \right]^2 \right\}.
\]
We claim that there exists $\lambda\geq 1$ such that $\mathscr{R}_{\lambda}$ is infinite.
To prove this, we first use a union bound and Reimer's inequality as in \eqref{eq:intrinsic_2nd_moment_Reimer} to obtain that
\begin{equation}
\bE_p \left[ |B_\mathrm{int}(u,r)|^2 \right] \leq \sum_{\ell=0}^r \bE_p \left[ | B_\mathrm{int}(u,\ell )| \right] \sup_{v \in V} \bE_p \left[ | B_\mathrm{int}(v,r-\ell )| \right]^2
\end{equation}
for every $u\in V$ and $r\geq 0$.
Taking the supremum over $u$, this inequality may then be rewritten in terms of $h_p$ and $\gamma_\mathrm{int}$ as
\begin{multline}
\sup_{v\in V} \bE_p \left[ |B_\mathrm{int}(v,r)|^2 \right] 
\leq
\sum_{\ell=0}^r \exp \left[\gamma_\mathrm{int}(p)\ell + 2 \gamma_\mathrm{int}(p) (r-\ell) + h_p(\ell)+2h_p(r-\ell) \right]
\\
=
\sup_{v\in V}\bE_p \left[ |B_\mathrm{int}(v,r)| \right]^2 \sum_{\ell=0}^r \exp \left[-\gamma_\mathrm{int}(p)\ell - (2h_p(r)-h_p(\ell)- 2h_p(r-\ell)) \right].
\end{multline}
We now split into two cases according to whether or not $h_p(r)$ is bounded as $r\to \infty$.
If $h_p$ is bounded by some constant $C_1$, then the sum on the right hand side of the last line is also bounded by the constant $C_2 = \sum_{\ell=0}^\infty \exp \left[-\gamma_\mathrm{int}(p)\ell + 3C_1 \right] $. Meanwhile, since $h_p$ is non-negative, we trivially have that $h_p(r) \geq \max_{0\leq \ell \leq r} h_p(\ell) - C_1$ for every $r\geq 1$ so that $\mathscr{R}_{C_1 \vee C_2}=\mathbb{N}$ is infinite in this case as claimed. On the other hand, if $h_p$ is not bounded, then the set of running maxima $\mathscr{R}' = \{r \geq 0 : h_p(r) = \max_{0 \leq \ell \leq r} h_p(r-\ell)\}$ must be infinite, and if $r \in \mathscr{R}'$ then
\begin{align}\bE_p \left[ |B_\mathrm{int}(v,r)|^2 \right] 
&\leq
\sup_{v\in V}\bE_p \left[ |B_\mathrm{int}(v,r)| \right]^2 \sum_{\ell=0}^r \exp \left[-\gamma_\mathrm{int}(p)\ell + h_p(\ell) \right] 
\nonumber\\
&\leq \sup_{v\in V}\bE_p \left[ |B_\mathrm{int}(v,r)| \right]^2 \sum_{\ell=0}^\infty \exp \left[-\gamma_\mathrm{int}(p)\ell + h_p(\ell)\right].
\end{align}
Since $\lim_{\ell\to \infty}\frac{1}{\ell}h_p(\ell)=0$, the series on the last line converges. Thus, if we set the constant $C_3$  to be $\sum_{\ell=0}^\infty \exp \left[-\gamma_\mathrm{int}(p)\ell + h_p(\ell)\right]$ then $\mathscr{R}_{C_3}$ contains $\mathscr{R}'$ and is therefore infinite since we assumed $h_p$ to be unbounded. This completes the proof of the claim.

\medskip

Fix $\lambda$
such that $\mathscr{R}_\lambda$ is infinite. Since $G$ is quasi-transitive, we have by the pigeonhole principle that there exists $v_0\in V$ such that
\begin{equation}\mathscr{R}_{\lambda,v_0}
=\mathscr{R}_\lambda \cap \left\{\bE_p \left[ |B_\mathrm{int}(v_0,r)| \right] = \sup_{v\in V} \bE_p \left[ |B_\mathrm{int}(v,r)| \right]\right\}
\end{equation}
is infinite also.  Note that if $r \in \mathscr{R}_{\lambda,v_0}$ then we have by the definitions that
\begin{equation}
\label{eq:sR_property}
e^{\gamma_\mathrm{int}(p)\ell} \sup_{v\in V} \bE_p \left[ |B_\mathrm{int}(v,r-\ell)| \right] = e^{\gamma_\mathrm{int}(p)r + h_p(r-\ell)} \leq e^{\gamma_\mathrm{int}(p)r + h_p(r) +\lambda} = e^\lambda \bE_p \left[ |B_\mathrm{int}(v_0,r)| \right]
\end{equation}
for every $0\leq \ell \leq r$.
 For each $\eps>0$, let 
 \[
R_\eps = \inf\Bigl\{r \geq 0 :|\partial B_\mathrm{int}(v_0,r)| \leq \eps e^{\gamma_\mathrm{int}(p)r} \Bigr\},
 \] 
where we take $\inf \emptyset = \infty$. It suffices to prove that there exists $\eps>0$ such that $\bP_p(R_\eps=\infty)>0$. Let $\cF_\eps$ be the $\sigma$-algebra generated by $R_\eps$ and $B_\mathrm{int}(v_0,R_\eps)$.
Conditional on $\cF_\eps$, we have for each $v \in \partial B_\mathrm{int}(v,R_\eps)$ that the set of $w\in V$ that are connected to $v$ by an open path of length at most $r-R_\eps$ that is disjoint from $B_\mathrm{int}(v,R_\eps)$ except at its endpoints is stochastically dominated by the unconditioned law of $B_\mathrm{int}(v,r-R_\eps)$. Thus for each $r \geq 0$ and $0 \leq \ell \leq r$ we have that
\begin{align}
\bE_p \left[ |B_\mathrm{int}(v_0,r)| \mid \cF_\eps \right] &\leq \begin{cases}
|B_\mathrm{int}(v_0,R_\eps)| + 
 \eps e^{\gamma_\mathrm{int}(p)R_\eps } \sup_{v\in V} \bE_p|B_\mathrm{int}(v,r-R_\eps)| & R_\eps \leq r\\
|B_\mathrm{int}(v_0,r)| & R_\eps > r
 \end{cases}
\end{align}
 and hence by \eqref{eq:sR_property} that
\begin{align}
\bE_p \left[ |B_\mathrm{int}(v_0,r)| \mid \cF_\eps \right] &\leq \begin{cases}
|B_\mathrm{int}(v_0,R_\eps)| + 
 \eps e^\lambda \bE_p\left[|B_\mathrm{int}(v_0,r)|\right] & R_\eps \leq r\\
|B_\mathrm{int}(v_0,r)| & R_\eps > r
 \end{cases}
\end{align}
for every $r \in \mathscr{R}_{\lambda,v_0}$ and $\eps>0$. Taking expectations, it follows that
\begin{align}
&\bE_p\left[|B_\mathrm{int}(v_0,r)| \mathbbm{1}(R_\eps \leq \ell) \right] 
\nonumber\\
&\hspace{3.5cm}\leq \bE_p\left[|B_\mathrm{int}(v_0,R_\eps)|\mathbbm{1}(R_\eps \leq \ell)\right] + 
 \eps  e^{\gamma_\mathrm{int}(p) R_\eps}\sup_{v\in V}\bE_p\left[|B_\mathrm{int}(v,r-R_\eps)|\right] \bP_p(R_\eps \leq \ell)
 \nonumber\\
&\hspace{3.5cm}\leq \bE_p\left[|B_\mathrm{int}(v_0,R_\eps)|\mathbbm{1}(R_\eps \leq \ell)\right] + 
 \eps e^\lambda \bE_p\left[|B_\mathrm{int}(v_0,r)|\right] \bP_p(R_\eps \leq \ell)
 \nonumber\\
&\hspace{3.5cm}\leq \bE_p\left[|B_\mathrm{int}(v_0,\ell)|\right] + 
 \eps e^\lambda \bE_p\left[|B_\mathrm{int}(v_0,r)|\right]
 \nonumber\\
&\hspace{3.5cm}\leq e^{\lambda}\left(e^{-\gamma_\mathrm{int}(p)(r-\ell)}+\eps\right) \bE_p\left[|B_\mathrm{int}(v_0,r)|\right]
\end{align}
for every $r \in \mathscr{R}_{\lambda,v_0}$, $\eps>0$, and $0\leq \ell \leq r$, where we applied \eqref{eq:sR_property} in the first and last inequalities.
On the other hand, we have by Cauchy-Schwarz and the definition of $\sR_{\lambda,v_0}$ that
\begin{align}
\bE_p \left[ |B_\mathrm{int}(v_0,r)| \mathbbm{1}(R_\eps > \ell) \right] &\leq \bE_p \left[ |B_\mathrm{int}(v_0,r)|^2  \right]^{1/2} \bP_p(R_\eps >\ell)^{1/2} 
\nonumber\\&\leq \lambda^{1/2}\bE_p\left[|B_\mathrm{int}(v_0,r)|\right]\bP_p(R_\eps >\ell)^{1/2}
\end{align}
for every $r\in \mathscr{R}_{\lambda,v_0}$, $\eps>0$, and $0\leq \ell \leq r$. Putting these two bounds together yields that
\begin{equation}
\bE_p \left[ |B_\mathrm{int}(v_0,r)|  \right] 
\leq \left(e^\lambda \eps + e^\lambda e^{-\gamma_\mathrm{int}(p)(r-\ell)} + \lambda^{1/2}\bP_p(R_\eps >\ell)^{1/2} \right)\bE_p\left[|B_\mathrm{int}(v_0,r)|\right]
\end{equation}
for every $r \in \mathscr{R}_{\lambda,v_0}$, $\eps>0$, and $0\leq \ell \leq r$. 
Rearranging, we deduce that
\begin{equation}
\bP_p(R_\eps >\ell)^{1/2} \geq \frac{1}{\lambda^{1/2}}\left[1- e^\lambda \eps + e^{\lambda-\gamma_\mathrm{int}(p)(r-\ell)} \right]
\end{equation}
for every $r \in \mathscr{R}_{\lambda,v_0}$, $\eps>0$, and $0\leq \ell \leq r$. Since $\mathscr{R}_{\lambda,v_0}$ is infinite and $\gamma_\mathrm{int}(p)$ is positive, it follows by taking the limit as $r\to \infty$ along $\mathscr{R}_{\lambda,v_0}$ that
\begin{equation}
\bP_p(R_\eps >\ell)^{1/2} \geq \frac{1}{\lambda^{1/2}}\left[1- e^\lambda \eps \right]
\end{equation}
for every $\eps>0$ and $\ell \geq 0$.
If $\eps < e^{-\lambda}$ then the right hand side is positive and does not depend on $\ell$, so that 
\begin{equation}\bP_p(R_\eps = \infty) = \lim_{\ell\to \infty} \bP_p(R_\eps >\ell) \geq \lambda^{-1}\left[1- e^\lambda \eps \right]^2>0\end{equation}
for every $\eps < e^{-\lambda}$. This completes the proof.
\end{proof}

\subsection{Positivity of the intrinsic growth on amenable graphs of exponential growth}

In this section we prove \cref{thm:amenable}.

\begin{proof}[Proof of \cref{thm:amenable}]
The case that $G$ is nonamenable follows from either \cite[Theorem 3.1]{BLS99} (yielding that the infinite cluster always contains a subgraph with positive Cheeger constant) or the results of \cite{HermonHutchcroftSupercritical} (since anchored expansion implies exponential growth). As such, it suffices to consider the case that $G$ is amenable, in which case the infinite cluster is unique for every $p>p_c$. Fix one such $p>p_c$. The Harris-FKG inequality implies that
\[
\bP_p(u \leftrightarrow v) \geq \bP_p(u \leftrightarrow \infty) \bP_p(v \leftrightarrow \infty) \geq \theta_*(p)^2
\]
for every $u,v\in V$, where we define $\theta_*(p)=\min_v \bP_p(v\leftrightarrow \infty)$. Since $G$ is quasi-transitive, it follows by continuity of measure that for each $r\geq 1$ there exists $R(r,p)<\infty$ such that
\[
\min \bigl\{\bP_p\bigl(u \xleftrightarrow{R(r,p)} v\bigr):u,v\in V, d(u,v)\leq r\bigr\} \geq \frac{1}{2}\theta_*(p)^2,
\]
where $\{u \xleftrightarrow{R(r,p)} v\}$ denotes the event that $u$ and $v$ are connected by a path of length at most $R(r,p)$. Note that if $u$ and $v$ have distance at most $kr$ then there exists a sequence $u=u_0,u_1,\ldots,u_k=v$ such that $d(u_i,u_{i+1})\leq r$ for each $0\leq i \leq k-1$, and if the events $\{u_i \xleftrightarrow{R(r,p)} u_{i+1}\}$ all hold for every $0\leq i \leq k-1$ then $u$ is connected to $v$ by an open path of length at most $kR(r,p)$.
Applying Harris-FKG again, we deduce that
\[
\min \bigl\{\bP_p\bigl(u \xleftrightarrow{kR(r,p)} v\bigr):u,v\in V, d(u,v)\leq kr\bigr\} \geq \left(\frac{1}{2}\theta_*(p)^2\right)^k
\]
for every $k,r\geq 1$. Letting $\gamma=\lim_{n\to\infty}\frac{1}{n} \log |B(v,n)|$ be the exponential growth rate of $G$, it follows that
\begin{align*}
\gamma_\mathrm{int}(p) &\geq \lim_{k\to\infty} \frac{1}{kR(r,p)} \log\left[ |B(v,kr)| \min \bigl\{\bP_p\bigl(u \xleftrightarrow{kR(r,p)} v\bigr):u,v\in V, d(u,v)\leq kr\bigr\}\right]
\\
&\geq \frac{r\gamma}{R(r,p)} - \frac{1}{R(r,p)} \log \frac{2}{\theta_*(p)^2}
\end{align*}
for every $r\geq 1$. The claim follows by taking $r$ sufficiently large that $r\gamma > \log \frac{2}{\theta_*(p)^2}$.
\end{proof}

\section{The anchored Cheeger constant}
\label{subsec:anchored_Cheeger}

In this section we prove \cref{thm:Cheeger}. We begin by establishing the upper bound.

\begin{lemma}
\label{lem:Cheeger_upper}
Let $G$ be a connected, locally finite, quasi-transitive graph, and suppose that $p_c<p_{2\to 2}$. Then there exists a constant $C$ such that for every $p_c < p \leq 1$, every infinite cluster $K$ in Bernoulli-$p$ bond percolation on $G$ has 
\[
\Phi^*(K) \leq C (p-p_c)^2
\]
$\bP_p$-almost surely.
\end{lemma}

Before proving this lemma, we first prove the following simple concentration lemma for the number of vertices in a set that belong to an infinite cluster. We define $\theta^*(p) := \sup_{v\in V} \bP_p(v \leftrightarrow \infty)$.

\begin{lemma}
\label{lem:Ainfty_variance}
Let $G=(V,E)$ be a countable graph, and let $0\leq p <p_{2\to 2}(G)$. Let $A \subset V$ be a finite set of vertices and let $A_\infty = \{v\in A : v\leftrightarrow \infty\}$. Then the variance of $|A_\infty|$ satisfies
\[
\mathbf{Var}_p |A_\infty| := \bE_p\left[\left(|A_\infty|-\bE_p\left[|A_\infty|\right]\right)^2\right] \leq \theta^*(p) \|T_p\|_{2\to 2}^2 |A|.\]
\end{lemma}

\begin{proof}[Proof of \cref{lem:Ainfty_variance}]
Consider the matrix $T_{p,\infty}\in [0,1]^{V^2}$ defined by setting
$T_{p,\infty}(u,v):=\bP_p(u \leftrightarrow v \text{ and } u \leftrightarrow \infty)$ for each $u,v \in V$. We claim that
\begin{equation}
\label{eq:Tpinfty}
T_{p,\infty} \opleq \theta^*(p) T_p^2 \qquad \text{ and hence that } \qquad \|T_{p,\infty}\|_{2\to 2} \leq \theta^*(p)\|T_{p}\|_{2\to 2}^2
\end{equation}
for every $0 \leq p \leq 1$, where $
\opleq$ denotes entrywise inequality of matrices. 
Indeed, if $u$ is connected to both $v$ and $\infty$ then there must exist a vertex $w$ (possibly equal to either $u$ or $v$) such that the event $\{u \leftrightarrow w\} \circ \{w \leftrightarrow \infty\} \circ \{w \leftrightarrow v\}$ occurs. Applying the BK inequality, it follows that
\begin{equation}
T_{p,\infty}(u,v) \leq \sum_w \bP_p(u \leftrightarrow w) \bP_p(w \leftrightarrow \infty) \bP_p(w \leftrightarrow v),
\end{equation}
which clearly implies the claimed inequality \eqref{eq:Tpinfty}.
We deduce that
\begin{align}
\bE_p|A_\infty|^2 &= \sum_{u,v \in A} \bP_p(u \leftrightarrow \infty, v \leftrightarrow \infty, u \nleftrightarrow v) + \sum_{u,v\in A} \bP_p(u \leftrightarrow \infty, v \leftrightarrow \infty, u \leftrightarrow v)
\nonumber\\
&\leq 
 \sum_{u,v \in A} \bP_p(u \leftrightarrow \infty) \bP_p(v \leftrightarrow \infty) + \sum_{u,v\in A} \bP_p(u \leftrightarrow \infty, u \leftrightarrow v)
 \nonumber\\
 &= \left(\bE_p|A_\infty|\right)^2 + \langle T_{p,\infty} \mathbbm{1}_A, \mathbbm{1}_A \rangle \leq \left(\bE_p|A_\infty|\right)^2 + \theta^*(p)\|T_p\|^2_{2\to 2} |A|
\end{align}
as required, where the inequality in the second line follows from the BK inequality.
\end{proof}

Given a set of vertices $K$ in $G$, we write $\partial^\omega_E K = \{ e \in \partial_E K : \omega(e)=1\}$ for the set of open edges belonging to the edge boundary of $K$. 

\begin{proof}[Proof of \cref{lem:Cheeger_upper}]
Let $\alpha$ be a constant to be chosen, let $p_0 = (p_c+p_{2\to 2})/2$, and fix $v\in V$. Since the inequality $\Phi^*(K)\leq 1$ holds vacuously, it suffices to prove the claim for $p_c<p \leq p_0$.
Fix one such $p_c < p \leq p_0$ and a vertex $v$ of $G$. By \cref{prop:growth_estimate,thm:KestenStigum} there exists a constant $C_1$ such that
\[
\prod_{i=0}^{r-1} \left(1+\frac{|\partial B_\mathrm{int}(v,i+1)|}{|B_\mathrm{int}(v,i)|}\right) = |B_\mathrm{int}(v,r)| \leq e^{C_1(p-p_c) r}
\]
for all sufficiently large $r$ almost surely. (Note that we are only using the easy parts of \cref{prop:growth_estimate,thm:KestenStigum} to reach this conclusion.)
 Rearranging, this implies that there exists a constant $C_2$ such that
\[
 \liminf_{r\to \infty} \frac{|\partial B_\mathrm{int}(v,r)|}{ |B_\mathrm{int}(v,r)|} \leq \frac{e^{C_1(p-p_c)}-1}{e^{C_1(p-p_c)}} \leq C_2(p-p_c)
\]
almost surely on the event that $v$ is in an infinite cluster. 


 We now perform a breadth-first search of the cluster of $v$: At stage $0$ we expose the value of every edge touching $v$. At each subsequent stage $i\geq 1$ we expose the value of those edges that touch $\partial B_\mathrm{int}(v,i-1)$ and have not yet been exposed, stopping if and when $\partial B_\mathrm{int}(v,i)=\emptyset$. Let $T_j$ be the $j$th time that $|\partial B_\mathrm{int}(v,r)| \leq 2C_2 (p-p_c) |B_\mathrm{int}(v,r)|$, so that $j \leq T_j<\infty$ for every $j\geq 1$ almost surely on the event that $v$ is in an infinite cluster. Let $\cF_i$ be the $\sigma$-algebra generated by the exploration up to time $i$, and let $\cF_{T_j}$ be the stopped $\sigma$-algebra generated by by the exploration up to time $T_j$. For each $i \geq 1$, let $X_i$ be the number of vertices of $B_\mathrm{int}(v,i)$ that are connected to infinity off of $B_\mathrm{int}(v,i)$, and note that any such vertex must belong to $\partial B_\mathrm{int}(v,i)$. Conditional on $T_j<\infty$ and on the stopped $\sigma$-algebra $\cF_{T_j}$, the expectation $\bE_p[X_{T_j} \mid \cF_{T_j}]$ is at most $\theta^*(p) |\partial B_\mathrm{int}(v,T_j)| \leq C_3 (p-p_c)^2 |B_\mathrm{int}(v,T_j)|$ for some constant $C_3$. By \cref{lem:Ainfty_variance} (applied to the subgraph of $G$ spanned by those edges that have not yet been queried by stage $T_j$), the conditional \emph{variance} of $X_{T_j}$ is at most $C_4 (p-p_c)^2 |B_\mathrm{int}(v,T_j)|$ for some constant $C_4$. It follows by Chebyshev's inequality that there exist positive constants $C_5$ and $C_6$ such that
 \[
\bP_p\left( X_{T_j} \geq C_5 (p-p_c)^2|B_\mathrm{int}(v,T_j)| \mid \cF_{T_j}\right) \leq \frac{C_6 \mathbbm{1}(T_j<\infty)}{(p-p_c)^2|B_\mathrm{int}(v,T_j)|}.
 \]
Since the right hand side tends to zero as $j\to\infty$, it follows by Fatou's lemma that
\[
\liminf_{i\to\infty} \frac{X_{i}}{|B_\mathrm{int}(v,i)|} \leq \liminf_{j\to\infty} \frac{X_{T_j}}{|B_\mathrm{int}(v,T_j)|} \leq C_5(p-p_c)^2
\]
almost surely on the event that $v$ is in an infinite cluster. Let $\operatorname{Hull}(v,i) \supseteq B_\mathrm{int}(v,i)$ be the set of all vertices $u$ in the cluster of $v$ such that any path from $u$ to $\infty$ in $K_v$ must pass through $B_\mathrm{int}(v,i)$. Then we have that $|\partial_E^\omega \operatorname{Hull}(v,i)| \leq M X_i$, where $M$ is the maximum degree of $G$, so that
\[
\liminf_{i\to\infty}\frac{|\partial_E^\omega \operatorname{Hull}(v,i)|}{|\operatorname{Hull}(v,i)|} \leq \liminf_{i\to\infty} \frac{M X_{i}}{|B_\mathrm{int}(v,i)|} \leq M C_5  (p-p_c)^2
\]
almost surely on the event that $v$ is in an infinite cluster. The claim follows since $v$ was arbitrary.
\end{proof}


Our final goal is to apply \cite[Theorem 1.1]{hutchcroft2020slightly} and \cite[Proposition 3.2]{HermonHutchcroftSupercritical} to complete the proof of \cref{thm:Cheeger}. The case of the inequality in which $p$ is very close to $1$ will require the following estimate on the exponential decay rate
\[
\zeta(p) := \liminf_{n\to\infty} -\frac{1}{n}\sup_{v\in V} \log \bP_p(n\leq |K_v| < \infty),
\] 
which is adapted from \cite[Theorem 2]{bperc96}.

\begin{lemma}
\label{lem:Benjamini_Schramm}
Let $G$ be a nonamenable locally finite graph with Cheeger constant $\Phi(G)>0$. Then
\begin{align}
\zeta(p) \geq \Phi(G) \log \frac{\Phi(G)}{1-p} + (1-\Phi(G)) \log \frac{\Phi(G)}{p} 
\end{align}
for every $0<p<1$.
\end{lemma}

Note that this bound is only positive for $p > 1-\Phi(G)$.

\begin{proof}[Proof of \cref{lem:Benjamini_Schramm}]
Let $(X_i)_{i\geq 1}$ be an i.i.d.\ sequence of Beroulli-$p$ random variables and let $v$ be a vertex of $G$. We can couple percolation on $G$ with the sequence $X_i$ so that the cluster of $v$ touches $\sum_{i=1}^n X_i$ open edges and $n-\sum_{i=1}^n X_i$ closed edges on the event that $|E(K_v)|=n$. The number of closed edges in the boundary of $K_v$ must be at least $\Phi(G) |E(K_v)|$, and it follows that
\begin{align}
\zeta(p)=\limsup_{n\to\infty} -\frac{1}{n} \log \bP_p(|E(K_v)|=n) &\geq \limsup_{n\to\infty} -\frac{1}{n} \log\P\left(\sum_{i=1}^n X_i \leq (1-\Phi(G)) n \right) 
\nonumber\\
&= \Phi(G) \log \frac{\Phi(G)}{1-p} + (1-\Phi(G)) \log \frac{\Phi(G)}{p} 
\end{align}
for every $0 < p < 1$, where the second line follows by standard large deviations theory (i.e., Cram\'er's theorem). 
\end{proof}

\begin{proof}[Proof of \cref{thm:Cheeger}] 

The upper bound is immediate from \cref{lem:Cheeger_upper}. On the other hand, \cite[Proposition 3.2]{HermonHutchcroftSupercritical}, which is based on an argument of Pete \cite[Theorem A.1]{chen2004anchored}, states that if $p>p_c$ then every infinite cluster $K$ of Bernoulli-$p$ bond percolation on $G$ has anchored expansion with anchored cheeger constant
\begin{equation}
\label{eq:zetatocheeger}
\Phi^*(K) \geq  \frac{1}{2} \sup\left\{ \alpha \in [0,p] : \alpha^{-\alpha} (1-\alpha)^{-(1-\alpha)} \left[\frac{p}{1-p}\right]^{\alpha} < e^{\zeta(p)} \right\} 
\end{equation}
almost surely,
where 
\[
\zeta(p) := \liminf_{n\to\infty} -\frac{1}{n}\sup_{v\in V} \log \bP_p(n\leq |K_v| < \infty)
\] 
for each $p\in [0,1]$.
It follows from \cite[Theorem 1.1 and Corollary 1.3]{hutchcroft2020slightly} that there exist positive constants $c$ and $\delta$ such that $\zeta(p) \geq c(p-p_c)^2$ for every $p_c < p \leq p_c+\delta$. Meanwhile, the main result of \cite{HermonHutchcroftSupercritical} states that $\zeta(p)>0$ for every $p_c < p \leq 1$, and it follows by continuity of $\zeta$ (see e.g. \cite[Theorem 10.1]{georgakopoulos2019exponential}) that there exists a constant $c_2$ such that $\zeta(p) \geq c_2$ for every $p_c+\delta \leq p \leq 1$. 
Putting these estimates together with \cref{lem:Benjamini_Schramm}, we deduce that there exists a positive constant $c_1$ such that
\begin{equation}
\zeta(p) \geq c_1 (p-p_c)^2 \log \frac{1}{1-p}
\end{equation}
for every $p_c < p < 1$. The claim follows from this and \eqref{eq:zetatocheeger} by direct calculation, since if $\alpha = c_2 (p-p_c)^2/\log 1/(p-p_c)$ for a sufficiently small positive constant $c_2$ then $\alpha \leq p$ and
\[
\alpha^{-\alpha} (1-\alpha)^{-(1-\alpha)} \left[\frac{p}{1-p}\right]^{\alpha} < e^{c_1 (p-p_c)^2 \log \frac{1}{1-p}}
\]
for every $p_c<p<1$.
\end{proof}

\section{Open problems}

Let us end the paper with some natural questions raised by our work. Some of these questions are similar in spirit to those raised by Benjamini, Lyons and Schramm in their 1999 work \cite{BLS99}, many of which remain open.


\begin{question}
Let $G$ be a nonamenable Cayley graph with $p_c<p_{2\to2}$ and for which the volume of $G$ has unbounded subexponential corrections to growth, such as $G = T \times \Z^d$. At which values of $p$ do the infinite clusters of $G$ have unbounded subexponential corrections to growth? Is the growth of clusters always either pure exponential or of the same form as $G$? If a transition from one behaviour to the other occurs, does it do so at $p_{2\to2}$, $p_u$, or some other point?
\end{question}

\begin{question}
Under the hypotheses of \cref{thm:KestenStigum}, do we have that 
\[
0<\liminf_{r\to\infty} \frac{|\partial B_\mathrm{int}(v,r)|}{\bE_p|\partial B_\mathrm{int}(v,r)|}\leq \limsup_{r\to\infty} \frac{|\partial B_\mathrm{int}(v,r)|}{\bE_p|\partial B_\mathrm{int}(v,r)|} < \infty
\]
almost surely on the event that $v$ belongs to an infinite cluster?
\end{question}

\begin{question}
Can the $1/\log1/(p-p_c)$ factor be removed from the lower bound of \cref{thm:Cheeger}?
\end{question}




\subsection*{Acknowledgments} This work was carried out in part while the author was a Senior Research Associate at the University of Cambridge, during which time he was supported by ERC  starting grant 804166 (SPRS).

\addcontentsline{toc}{section}{References}

 \setstretch{1}
  \bibliographystyle{abbrv}
  \bibliography{unimodularthesis.bib}
\end{document}